\documentclass[11pt]{amsart}

\usepackage{amsfonts,epsfig}
\usepackage{latexsym}
\usepackage{amssymb}
\usepackage{amsmath}
\usepackage{amsthm}
\usepackage[pdfencoding=auto,psdextra,backref]{hyperref}
\usepackage{graphics}
\usepackage[table]{xcolor}
\usepackage{verbatim}
\usepackage{enumerate}
\usepackage[all]{xy}
\usepackage{array,booktabs,ctable}
\usepackage{color}
\definecolor{mylinkcolor}{rgb}{0.8,0,0}
\definecolor{myurlcolor}{rgb}{0,0,0.8}
\definecolor{mycitecolor}{rgb}{0,0,0.8}
\hypersetup{colorlinks=true,urlcolor=myurlcolor,citecolor=mycitecolor,linkcolor=mylinkcolor,linktoc=page,breaklinks=true}
\usepackage{bigstrut}

\newtheorem{defn}{Definition}[section]

\newtheorem{lemma}[defn]{Lemma}

\newtheorem{thm}[defn]{Theorem}
\newtheorem{theorem}[defn]{Theorem}
\newtheorem{cor}[defn]{Corollary}
\newtheorem{prop}[defn]{Proposition}
\newtheorem{proposition}[defn]{Proposition}

\theoremstyle{definition}

\newtheorem{remark}[defn]{Remark}
\newtheorem{question}[defn]{Question}
\newtheorem{example}[defn]{Example}

\newcommand{\QQ}{\mathbb Q}
\newcommand{\Qbar}{\overline{\QQ}}
\newcommand{\ZZ}{\mathbb Z}

\newcommand{\FF}{\mathbb F}

\newcommand{\PP}{\mathbb P}

\newcommand{\Gal}{\operatorname{Gal}}
\newcommand{\Aut}{\operatorname{Aut}}

\newcommand{\GL}{\operatorname{GL}}

\renewcommand{\Im}{\operatorname{Im}}

\newcommand{\tor}{\mathrm{tors}}

\begin{document}



\bibliographystyle{plain}
\title[Torsion over $\QQ(D_4^\infty)$]{Torsion subgroups of rational elliptic curves over the compositum of all $D_4$ extensions of the rational numbers}

\author{Harris B. Daniels}
\address{Department of Mathematics and Statistics, Amherst College, Amherst, MA 01002, USA}
\email{hdaniels@amherst.edu}
\urladdr{http://hdaniels.people.amherst.edu}

\keywords{Elliptic Curve, Torsion Points, Galois Theory}

\subjclass[2010]{Primary: 11G05, Secondary: 11R21, 12F10, 14H52.}

\begin{abstract}
Let $E/\QQ$ be an elliptic curve and let $\QQ(D_4^\infty)$ be the compositum of all extensions of $\QQ$ whose Galois closure has Galois group isomorphic to a quotient of a subdirect product of a finite number of transitive subgroups of $D_4$. In this article we first show that $\QQ(D_4^\infty)$ is infact the compostium of all $D_4$ extensions of $\QQ$ and then we prove that the torsion subgroup of $E(\QQ(D_4^\infty))$ is finite and determine the 24 possibilities for its structure. We also give a complete classification of the elliptic curves that have each possible torsion structure in terms of their $j$-invariants.
\end{abstract}

\maketitle

\section{Introduction}

A fundamental theorem in arithmetic geometry known as the Mordell--Weil theorem says that the rational points on an elliptic curve defined over a number field can be given the algebraic structure of a finitely generated abelian group. More specifically, if $K$ is a number field and $E/K$ is an elliptic curve then the set of $K$-rational points $E(K)$ is isomorphic to a group of the form $\ZZ^r\oplus E(K)_\tor$ where $r$ is a nonnegative integer and $E(K)_\tor$ is a finite abelian group called the {\it{torsion subgroup of $E$ over $K$}}. In fact, as long as the base field $K$ is a number field, the torsion subgroup of an elliptic curve is always isomorphic to a group of the form $\ZZ/a\ZZ\oplus \ZZ/ab\ZZ$ for some positive integers $a$ and $b$. Merel, in \cite{merel} proved the existence of a uniform bound on the size of $E(K)_\tor$ that depends only on the degree of the extension $K/\QQ$. In light of this result it is natural to ask the following question. 

\begin{question}\label{question:torsionsubgroups}
For a fixed $d\geq 1$, what groups (up to isomorphism) arise as the torsion subgroup of an elliptic curve over a number field of degree $d$?
\end{question}

The following theorems give a complete answer to Question \ref{question:torsionsubgroups} when $d=1$ and $2$.

\begin{thm}[Mazur \cite{mazur1}]\label{thm:mazur}
Let $E/\QQ$ be an elliptic curve. Then
\[
E(\QQ)_\tor \simeq \begin{cases}
 \ZZ/M\ZZ  &   1\leq M \leq 10 \hbox{ or } M=12, \hbox{ or}\\
\ZZ/2\ZZ\oplus\ZZ/2M\ZZ  &  1\leq M\leq 4. 
 \end{cases}
\]
\end{thm} 
\begin{thm}[Kenku, Momose \cite{kenmom}, Kamienny \cite{kamienny}]\label{thm:quadgroups}
Let $E/F$ be an elliptic curve over a quadratic number field $F$. Then
\[
E(F)_\tor\simeq
\begin{cases}
\ZZ/M\ZZ &\text{with}\ 1\leq M\leq 16\ \text{or}\ M=18,\ \text{or}\\
\ZZ/2\ZZ\oplus \ZZ/2M\ZZ &\text{with}\ 1\leq M\leq 6,\ \text{or}\\
\ZZ/3\ZZ \oplus \ZZ/3M\ZZ &\text{with}\ M=1\ \text{or}\ 2,\ \text{only if}\ F = \QQ(\sqrt{-3}),\ \text{or}\\
\ZZ/4\ZZ \oplus \ZZ/4\ZZ &\text{only if}\ F = \QQ(\sqrt{-1}).
\end{cases}
\]
\end{thm}

Recently, Etropolski, Morrow, and Zureick-Brown, (and independently Derickx) announced that they have found a complete classification of the torsion structures that occur for elliptic curves over cubic fields, giving an answer to Question \ref{question:torsionsubgroups} in the case when $d=3$. This answer comes after the work of Jeon, Kim, Lee, and Schweizer \cite{jeon1,jeon2} classifing all of the torsion structures of elliptic curves defined over cubic fields that occur infinitely often.


One way to simplify Question \ref{question:torsionsubgroups} is to restrict to the torsion subgroups of elliptic curves defined over $\QQ$ that have been base-extended to degree $d$ number fields. There are many results related to this question and below we present a few of them.

\begin{thm}\cite[Thm.~2]{najman}
\label{thm-najman1} Let $E/\QQ$ be an elliptic curve and let $F$ be a quadratic number field. Then
\[
E(F)_\tor\simeq
\begin{cases}
\ZZ/M\ZZ &\text{with}\ 1\leq M\leq 10 \text{ or } M = 12,15,16,\ \text{or}\\
\ZZ/2\ZZ\oplus \ZZ/2M\ZZ &\text{with}\ 1\leq M\leq 6,\ \text{or}\\
\ZZ/3\ZZ \oplus \ZZ/3M\ZZ &\text{with}\ 1\leq M\leq 2 \text{ and }F = \QQ(\sqrt{-3}),\ \text{or}\\
\ZZ/4\ZZ \oplus \ZZ/4\ZZ &\text{with}\ F = \QQ(\sqrt{-1}).
\end{cases}
\]

\end{thm}

\begin{thm}\cite[Thm.~1]{najman}
\label{thm-najman2} Let $E/\QQ$ be an elliptic curve and let $F$ be a cubic number field. Then
\[
E(F)_\tor\simeq
\begin{cases}
\ZZ/M\ZZ &\text{with}\ 1\leq M\leq 10 \text{ or } M = 12,13,14,18,21,\  \text{or}\\
\ZZ/2\ZZ\oplus \ZZ/2M\ZZ &\text{with}\ 1\leq M\leq 4 \text{ or } M=7.
\end{cases}
\]
Moreover, the elliptic curve $162B1$ over $\QQ(\zeta_9)^+$ is the unique rational elliptic curve over a cubic field with torsion subgroup isomorphic to $\ZZ/21\ZZ$. For all other groups $T$ listed above there are infinitely many $\Qbar$-isomorphism classes of elliptic curves $E/\QQ$ for which $E(F)\simeq T$ for some cubic field $F$.
\end{thm}

Another variant of this question would be to consider what torsion structures occur when base-extending an elliptic curve $E/\QQ$ to some fixed \emph{infinite} extension. Of course, since infinite extensions are not number fields the Mordell--Weil theorem does not immediately give us that the torsion subgroup an elliptic curve base-extended to an infinite extension is finite. Because of this we will need to carefully choose our infinite extensions so this question still has some content. 

\begin{defn}\label{def:Q(d)}
For each fixed integer $d\ge 1$, let $\QQ(d^\infty)$ denote the compositum of all field extensions $F/\QQ$ of degree $d$.
More precisely, let $\Qbar$ be a fixed algebraic closure of $\QQ$, and define
\[
\QQ(d^\infty) := \QQ\left(\{\beta\in \Qbar: [\QQ(\beta):\QQ]=d  \}\right).
\]
\end{defn}

The fields defined in Definition \ref{def:Q(d)} have been studied on their own by Gal and Grizzard \cite{grizzard}, and the torsion structures of the rational elliptic curves base-extended to $\QQ(2^\infty)$ and $\QQ(3^\infty)$ are fully classified.

\begin{thm}[Laska, Lorenz \cite{laska}, Fujita \cite{fujita1,fujita2}]\label{thm-fujita} Let $E/\QQ$ be an elliptic curve. The torsion subgroup $E(\QQ(2^\infty))_\tor$ is finite, and
\[
E(\QQ(2^\infty))_\tor\simeq
\begin{cases}
\ZZ/M\ZZ &\text{with}\ M\in 1,3,5,7,9,15,\ \text{or}\\
\ZZ/2\ZZ\oplus \ZZ/2M\ZZ &\text{with}\ 1\leq M\leq 6\ \text{or}\ M=8,\ \text{or}\\
\ZZ/3\ZZ\oplus \ZZ/3\ZZ & \text{or}\\
\ZZ/4\ZZ\oplus \ZZ/4M\ZZ &\text{with}\ 1\leq M\leq 4,\ \text{or}\\
\ZZ/2M\ZZ\oplus \ZZ/2M\ZZ &\text{with}\ 3\leq M\leq 4.\\
\end{cases}
\]
\end{thm}

\begin{thm}[D., Lozano-Robledo, Najman, Sutherland \cite{Q(3)}]\label{thm:main_Q(3)} Let $E/\QQ$ be an elliptic curve. The torsion subgroup $E(\QQ(3^\infty))_\tor$ is finite, and
\[
E(\QQ(3^\infty))_\tor\simeq
\begin{cases}
\ZZ/2\ZZ\oplus \ZZ/2M\ZZ &\text{with}\ M=1,2,4,5,7,8,13, \text{ or}\\
\ZZ/4\ZZ\oplus \ZZ/4M\ZZ &\text{with}\ M=1,2,4,7,\ \text{or}\\
\ZZ/6\ZZ\oplus \ZZ/6M\ZZ &\text{with}\ M=1,2,3,5,7,\ \text{or}\\
\ZZ/2M\ZZ\oplus \ZZ/2M\ZZ &\text{with}\ M=4,6,7,9.\\
\end{cases}
\]
All but $4$ of the torsion subgroups $T$ listed above occur for infinitely many $\Qbar$-isomorphism classes of elliptic curves $E/\QQ$; for $T= \ZZ/4\ZZ\times\ZZ/28\ZZ,\ \ZZ/6\ZZ\times\ZZ/30\ZZ,\ \ZZ/6\ZZ\times\ZZ/42\ZZ$, and $\ZZ/14\ZZ\times\ZZ/14\ZZ$ there are only $2$, $2$, $4$, and $1$ (respectively) $\Qbar$-isomorphism classes of $E/\QQ$ for which $E(\QQ(3^\infty))_\tor\simeq T$.
\end{thm}

In their article \cite{grizzard}, Gal and Grizzard prove a more general version of the following proposition.

\begin{prop}\cite[Theorem 9]{grizzard}\label{prop:GG-subfield}
Let $K/\QQ$ be a finite extension. Then $K\subseteq \QQ(d^\infty)$ if and only if the following two conditions are met. 
\begin{enumerate}
\item There exists a group $H$ which is a subdirect product of transitive subgroups of degree $d$ with some normal subgroup $N$ such that 
$$1\to N\to H \to \Gal(K/\QQ) \to 1$$
is a short exact sequence.
\item We can solve the corresponding Galois embedding problem, i.e. we can find a field $L\supseteq K$ such that $\Gal(L/\QQ)\simeq H$.
\end{enumerate}
\end{prop}

This proposition gives a concrete way to check is a given finite extension is contained inside of $\QQ(d^\infty)$ explicitly. Motivated by part (1) of this proposition we give the following definition which allows us to define some related fields. 

\begin{defn}\label{def:Q(G^infty)}
Let $G$ be a transitive subgroup of $S_d$ for some $d\geq 2$. We say that a finite group $H$ is of {\bfseries generalized $G$-type} if it is isomorphic to a quotient of a subdirect product of transitive subgroups of $G$. Given a number field $K/\QQ$ and its Galois closure $\widehat{K}$, we say that $K/\QQ$ is of {\bfseries generalized $G$-type} if $\Gal(\widehat{K}/\QQ)$ is a group of generalized $G$-type. Let $\QQ(G^\infty)$ be the compositum of all fields that are of generalized $G$-type. 
\end{defn}

While in general it is not necessarily true that $\QQ(S_d^\infty) =  \QQ(d^\infty)$ since the relevant Galois embedding problems are not always solvable, it is always true that $\QQ(d^\infty) \subseteq  \QQ(S_d^\infty)$. In the case that $d=3$, \cite{Q(3)} shows that in fact $\QQ(3^\infty)= \QQ(S_3^\infty)$. Further, if $G\subseteq S_d$ is transitive and nilpotent, then by a celebrated result of Shafarevich, the relevant Galois embedding problems are alway solvable and $\QQ(G^\infty)\subseteq \QQ(d^\infty)$ since the kernel of the embedding problem would also have to be nilpotent. In fact, $G$ needn't be nilpotent itself as long as all the normal proper subgroups of $G$ are nilpotent. For more information about what is known  regarding Galois embedding problems, including a complete statement and proof of Shafarevich's theorem on the solvability of Galois embedding problems the reader is encouraged to see either \cite{EmbeddingProblem}, \cite{Inverse Galois Theory}, or \cite{serreTopics}. 

\begin{thm}\label{thm-D4=F}
If $F/\QQ$ is the compositum of all $D_4$-extensions of $\QQ$, then $F = \QQ(D_4^\infty)$.
\end{thm}

\begin{proof}
Inspecting Definition \ref{def:Q(G^infty)} we can see that the field $\QQ(D_4^\infty)$ is the compositum of all $D_4$, $\ZZ/4\ZZ$, and $V_4$ extensions of $\QQ$. Here $V_4$ is the Klein Four-group. From this we immediately get that $F\subseteq \QQ(D_4^\infty)$ and to show the reverse inclusion we simply need to show that any $\ZZ/4\ZZ$ or $V_4$ extension of $\QQ$ is in $F$. 

Suppose that $K/\QQ$ is a finite extension such that $\Gal(K/\QQ)\simeq \ZZ/4\ZZ$ and consider the group $D_4\times D_4$ together with $t_i$ and $s_i$ generators of the $i$th copy of $D_4$ with order 4 and 2 respectively. Next let $G = \langle t_1s_2,t_2s_1  \rangle \subseteq D_4\times D_4$. The group $G$ has the property that $G/[G,G] \simeq \ZZ/4\ZZ\times \ZZ/2\ZZ$. Clearly, there exists a $K'/\QQ$ such that $\Gal(K'/\QQ)\simeq \ZZ/4\ZZ\times\ZZ/2\ZZ$ since we can just pick $K' = K(\sqrt{p})$ for some prime $p$ unramified in $K/\QQ$. This give us a short exact sequence 
$$\xymatrix{
1\ar[r]  & [G,G] \ar[r]& G\ar[r] & \Gal(K'/\QQ)\ar[r] & 1
}$$
From the celebrated result of Shafarevich mentioned above together with the fact that $[G,G]$ is nilpotent we know that there is a field $L\subseteq K' \subseteq K$ such that $\Gal(L/\QQ)\simeq G$. Studying the group $G$, we see that there are two normal subgroups $N_1= \langle t_1^2\rangle$ and $N_2=\langle t_1^2 \rangle$ such that $G/N_1 \simeq G/N_2\simeq D_4$ which means that there are two subfields of $L$, call them $L_1$ and $L_2$ such that $$\Gal(L_1/\QQ)\simeq\Gal(L_2/\QQ)\simeq D_4.$$ 
Further, since $|N_1|=|N_2|=2$ and $N_1\neq N_2$, we know that $[L:L_1] = [L:L_2] = 2$ and $L_1\neq L_2$. Therefore $L= L_1L_2\subseteq F$ and $K \subseteq K' \subseteq L = L_1L_2\subseteq F$. 

To see that every $V_4$ extension of $\QQ$ is contained in $F$ it is enough to note that by the same result of Shafarevich, every quadratic extension $F/\QQ$ can be embedded in an $\ZZ/4\ZZ$ extension of $\QQ$ and we just saw that every such extension is in $F$. 

\end{proof}

\begin{remark}
In light of Theorem \ref{thm-D4=F}, for the rest of the paper, we will use $\QQ(D_4^\infty)$ to denote both the compositum of $D_4$ extensions of $\QQ$ and well as the compositum of all generalized $D_4$ extensions of $\QQ$. 
\end{remark}

The main goal of this article is to classify (up to isomorphism) the groups that occur as the torsion subgroup of an elliptic curve defined over $\QQ$ and then base-extended to $\QQ(D_4^\infty)$. The reason we chose to work with the group $D_4$ is that it is the simplest nonabelain transitive subgroup of $S_4$ that has yet to be dealt with. We chose to work with a nonabelain group since the torsion structures of rational elliptic curves base-extended to the maximal abelian extension of $\QQ$ have been completely classified in \cite{chou2}.

\begin{thm}\label{thm:main}
Let $E/\QQ$ be an elliptic curve. The torsion subgroup $E(\QQ(D_4^\infty))_\tor$ is finite and 
\[
E(\QQ(D_4^\infty))_\tor\simeq
\begin{cases}
\ZZ/M\ZZ & \hbox{with }M = 1,3,5,7,9,13,15, \hbox{ or}\\
\ZZ/3\ZZ \oplus \ZZ/3M\ZZ & \hbox{with }M = 1,5 \hbox{ or}\\
\ZZ/4\ZZ \oplus \ZZ/4M\ZZ & \hbox{with }1\leq M\leq 6 \hbox{ or } M=8, \hbox{ or} \\
\ZZ/5\ZZ \oplus \ZZ/5\ZZ & \hbox{ or} \\
\ZZ/8\ZZ \oplus \ZZ/8M\ZZ&  \hbox{with }1\leq M\leq 4 \hbox{ or} \\
\ZZ/12\ZZ \oplus \ZZ/12M\ZZ&  \hbox{with }1\leq M\leq 2 \hbox{ or} \\
\ZZ/16\ZZ \oplus \ZZ/16\ZZ.&\\
\end{cases}
\]
All but $3$ of the $24$ torsion structures listed above occur for infinitely many $\Qbar$-isomorphism classes of elliptic curves $E/\QQ$. The torsion structures that occur finitely often are $\ZZ/15\ZZ$, $\ZZ/3\ZZ \oplus \ZZ/15\ZZ$, and $\ZZ/12\ZZ\oplus\ZZ/24\ZZ$ which occur for $4$, $2$, and $1$ $\Qbar$-isomorphism classes respectively. 
\end{thm}

In Table \ref{tab:Examples}, for each torsion structure $T$ that occurs in Theorem \ref{thm:main} we give the Cremona label of the elliptic curve $E/\QQ$ with the smallest conductor such that $E(\QQ(D_4^\infty))_\tor \simeq T$. Throughout the paper, we will refer to elliptic curves by their Cremona label \cite{cremona} and provide a hyperlink to their entry on the $L$-functions and Modular Forms Database \cite{lmfdb}.

\begin{table}[h!]
\begin{center}
\renewcommand{\arraystretch}{1}
\begin{tabular}{|l|l||l|l|}\hline
	$E/\QQ$ & $E(\QQ(D_4^\infty))_\tor$ & $E/\QQ$ & $E(\QQ(D_4^\infty))_\tor$\\
	\hline\bigstrut[t]
	\href{http://www.lmfdb.org/EllipticCurve/Q/26b2}{\texttt{26b2}}     & $\{\mathcal{O}\}$ &
    \href{http://www.lmfdb.org/EllipticCurve/Q/15a5}{\texttt{15a5}}   & $\ZZ/4\ZZ\oplus \ZZ/16\ZZ$ \\
	\href{http://www.lmfdb.org/EllipticCurve/Q/19a2}{\texttt{19a2}}     & $\ZZ/3\ZZ$ &
    \href{http://www.lmfdb.org/EllipticCurve/Q/66c1}{\texttt{66c1}}     & $\ZZ/4\ZZ\oplus \ZZ/20\ZZ$ \\
    \href{http://www.lmfdb.org/EllipticCurve/Q/11a2}{\texttt{11a2}}     & $\ZZ/5\ZZ$ &
    \href{http://www.lmfdb.org/EllipticCurve/Q/30a1}{\texttt{30a1}}     & $\ZZ/4\ZZ\oplus \ZZ/24\ZZ$ \\
	\href{http://www.lmfdb.org/EllipticCurve/Q/26b1}{\texttt{26b1}}     & $\ZZ/7\ZZ$ &
	\href{http://www.lmfdb.org/EllipticCurve/Q/210e1}{\texttt{210e1}}     & $\ZZ/4\ZZ\oplus \ZZ/32\ZZ$ \\
	\href{http://www.lmfdb.org/EllipticCurve/Q/54a2}{\texttt{54a2}}     & $\ZZ/9\ZZ$ &
	\href{http://www.lmfdb.org/EllipticCurve/Q/11a1}{\texttt{11a1}}     & $\ZZ/5\ZZ\oplus \ZZ/5\ZZ$ \\
	\href{http://www.lmfdb.org/EllipticCurve/Q/2890d1}{\texttt{2890d1}} & $\ZZ/13\ZZ$ &
	\href{http://www.lmfdb.org/EllipticCurve/Q/17a1}{\texttt{17a1}}   & $\ZZ/8\ZZ\oplus \ZZ/8\ZZ$ \\
	\href{http://www.lmfdb.org/EllipticCurve/Q/50a1}{\texttt{50a1}}     & $\ZZ/15\ZZ$ &
	\href{http://www.lmfdb.org/EllipticCurve/Q/15a2}{\texttt{15a2}}     & $\ZZ/8\ZZ\oplus \ZZ/16\ZZ$ \\
	\href{http://www.lmfdb.org/EllipticCurve/Q/19a1}{\texttt{19a1}}     & $\ZZ/3\ZZ\oplus \ZZ/3\ZZ$ &
	\href{http://www.lmfdb.org/EllipticCurve/Q/30a2}{\texttt{30a2}}     & $\ZZ/8\ZZ\oplus \ZZ/24\ZZ$ \\
	\href{http://www.lmfdb.org/EllipticCurve/Q/338d1}{\texttt{338d1}}   & $\ZZ/3\ZZ\oplus \ZZ/15\ZZ$ &
    \href{http://www.lmfdb.org/EllipticCurve/Q/210e2}{\texttt{210e2}} & $\ZZ/8\ZZ\oplus \ZZ/32\ZZ$ \\
	\href{http://www.lmfdb.org/EllipticCurve/Q/46a1}{\texttt{46a1}}   & $\ZZ/4\ZZ\oplus \ZZ/4\ZZ$ &
	\href{http://www.lmfdb.org/EllipticCurve/Q/14a1}{\texttt{14a1}}     & $\ZZ/12\ZZ\oplus \ZZ/12\ZZ$ \\
	\href{http://www.lmfdb.org/EllipticCurve/Q/17a3}{\texttt{17a3}}   & $\ZZ/4\ZZ\oplus \ZZ/8\ZZ$ &
	\href{http://www.lmfdb.org/EllipticCurve/Q/256a1}{\texttt{256a1}}     & $\ZZ/12\ZZ\oplus \ZZ/24\ZZ$ \\
	\href{http://www.lmfdb.org/EllipticCurve/Q/14a3}{\texttt{14a3}}   & $\ZZ/4\ZZ\oplus \ZZ/12\ZZ$ &
	\href{http://www.lmfdb.org/EllipticCurve/Q/15a1}{\texttt{15a1}}     & $\ZZ/16\ZZ\oplus \ZZ/16\ZZ$ \\
	\hline
\end{tabular}
\end{center}
\caption{Examples of minimal conductor for each possible torsion structure over $\QQ(D_4^\infty)$}  \label{tab:Examples}
\end{table}

All of the computations were done using Magma \cite{magma} and the procedures used to verify the examples in Table \ref{tab:Examples} and all the other computational results are available at \cite{MagmaCode}. These include the modular curves that were used to compute a complete set of rational functions that parametrize the elliptic curves over $\QQ$ with a given torsion structure when base-extended to $\QQ(D_4^\infty)$ in Table \ref{tab:jMaps}.

Many of the results in this article rely on recent advances on Galois representations associated to elliptic curves defined over $\QQ$. In particular, we take advantage of the impressive results of Rouse and Zureick-Brown in \cite{RZB}, where all the possible 2-adic images of Galois representations attached to elliptic curves are classified. We also use the equally impressive classifications of mod $p$ representations attached to rational elliptic curves given in \cite{zywina1} and the classification of modular curves of primes-power level with infinitely many points given in \cite{SZ}.

The organization of the paper is as follows. In Section \ref{sec:galoisreps}, we briefly review results about Galois representations attached to elliptic curves which will be useful for the rest of the paper. 
In Section \ref{sec:D4GroupsAndFields}, we find necessary and sufficient conditions for a finite group to be of generalized $D_4$-type and classify some important subfields of $\QQ(D_4^\infty)$. 
In Section \ref{sec:GrowthOfTorsion}, we give some general results about the field of definition of isogenies and how torsion groups grow under base-extension. Section \ref{sec:pPrimary} classifies all of the possible nontrivial $p$-primary components of torsion subgroups of an elliptic curve $E(\QQ(D_4^\infty))_\tor$. 
In Section \ref{sec:TorsionStructures}, we prove that the 24 groups in Theorem \ref{thm:main} are the only groups up to isomorphism that occur. 
Finally, in Section \ref{sec:Tparam}, we precisely characterize the elliptic curves that realize each of the subgroups listed in Theorem \ref{thm:main}.

\subsection{Acknowledgements} The author would like to thank \'Alvaro Lozano-Robledo for many useful conversations during the preparation of this paper. We would also like to thank Maarten Derickx, Jeffrey Hatley, Filip Najman, Andrew Sutherland, and the anonymous referee for useful comments and conversations about earlier versions of this article.

\section{Galois representations associated to elliptic curves}\label{sec:galoisreps}

Before talking about Galois representations associated to elliptic curves we need to establish some notation. We will follow the notational conventions laid out in \cite[Section 2]{Q(3)}, but repeat some of them here for the ease of the reader.

For the rest of the paper we fix an algebraic closure $\Qbar$ of $\QQ$ that contains all of the algebraic extensions of $\QQ$ that we are going to consider. For an elliptic curve $E$ defined over a field $K$, we let 
$$E[n] = \{ P\in E(\overline{K}) : nP = \mathcal{O}\}$$
where $\overline{K}$ is a fixed algebraic closure of $K$. A classical result in the study of rational elliptic curves is that $E[n] \simeq \ZZ/n\ZZ \oplus \ZZ/n\ZZ$ as long as the characteristic of $K$ is relatively prime to $n$. If $L/K$ is a field extension, we write $E(L)[n]$ for the $n$-torsion subgroup of $E(L)$ and $E(L)(p)$ for the $p$-primary component of $E(L)$. Given a point $P\in E(L)$, we write $K(P)$ for the extension of $K$ generated by adjoining the coordinates of $P$ and $K(x(P))$ for the extension of $K$ generated by adjoining the $x$-coordinate of $P$. We point out here that while $x(P)$ itself depends on the choice of model for $E$, the field $K(x(P))$ is in fact independent of this choice.

Given an elliptic curve $E/K$, an $n$-\emph{isogeny} is a cyclic isogeny $\varphi\colon E\to E'$ of degree $n$; this means $\ker\varphi$ is a cyclic subgroup of $E[n]$, and as all the isogenies we consider are separable, this cyclic group has order $n$. The isogenies $\varphi$ that we consider are also \emph{rational}, meaning that $\varphi$ is defined over~$K$, equivalently, that $\ker\varphi$ is \emph{Galois-stable}: the action of $\Gal(\overline{K}/K)$ on $E[n]$ given by its action on the coordinates of the points $P\in E[n]$ permutes $\ker\varphi$. We consider two (separable) isogenies to be distinct or nonisomorphic only when their kernels are distinct.

If $E/\QQ$ is an elliptic curve and $n$ is a positive integer, then $\Gal(\Qbar/\QQ)$ acts on $E[n]$ component-wise and this action induces a continuous homomorphism
\[
\bar\rho_{E,n}\colon \Gal(\Qbar/\QQ)\to \Aut(E[n])\simeq \GL_2(\ZZ/n\ZZ),
\]
whose image we will view as a subgroup of $\GL_2(\ZZ/n\ZZ)$ up to conjugacy. The reason it is only defined up to conjugacy is that the isomorphism between $\Aut(E[n])$ and $\GL_2(\ZZ/n\ZZ)$ and depends on a choice of basis for $E[n]$. The homomorphism $\bar\rho_{E,n}$ is called the {\it mod $n$ Galois representation attached to $E$}. The extension $\QQ(E[n])/\QQ$ is Galois and the restriction of $\bar\rho_{E,n}$ to $\Gal(\QQ(E[n])/\QQ)$ is injective and so $\Gal(\QQ(E[n])/\QQ)$ is isomorphic to a subgroup of $\GL_2(\ZZ/n\ZZ)$; namely, $\Im\bar\rho_{E,n}$. The determinant map $\det\colon \Im\bar\rho_{E,n} \to (\ZZ/n\ZZ)^\times$ must be surjective and there must be an element in $\Im\bar\rho_{E,n}$ (corresponding to complex conjugation) with trace $0$ and determinant $-1$. For more information about Galois representations attached to elliptic curves the reader should consult \cite{serre}.

Given that the Galois groups that we are interested in studying are isomorphic to subgroups of $\GL_2(\ZZ/n\ZZ)$, it will be worthwhile to distinguish some subgroups of $\GL_2(\ZZ/n\ZZ)$ up to conjugation. In particular, two groups are worth highlighting:
\begin{enumerate}
\item the {\it{Borel group}} of upper triangular matrices,
\item the {\it{split Cartan}} subgroup of diagonal matrices.
\end{enumerate}
We mention these two groups because if $E/\QQ$ is an elliptic curve, then $E$ has a rational $n$-isogeny if and only if $\Im\bar\rho_{E,n}$ is conjugate to a subgroup of the Borel subgroup of $\GL_2(\ZZ/n\ZZ)$. Similarly, $E$ admits two independent $n$-isogenies if and only if $\Im\bar\rho_{E,n}$ is conjugate to a subgroup of the split Cartan subgroup of $\GL_2(\ZZ/n\ZZ)$.

\section{Groups and Fields of Generalized  \texorpdfstring{$D_4$}{D4}-type}\label{sec:D4GroupsAndFields}

In this section we study groups and fields of generalized $D_4$-type. Recall that a group is said to be of {\it{generalized $D_4$-type}} if it is isomorphic to a quotient of a subdirect product of transitive subgroups of $D_4$.

\begin{example}
Clearly the groups $\ZZ/4\ZZ$, $\ZZ/2\ZZ$ are all of generalized $D_4$-type. More interestingly, the quaternion group $Q_8$ is generalized $D_4$-type since $Q_8\simeq G/H$ with 
$$G = \langle(2, 4)(5, 6, 7, 8), (1, 2, 3, 4) \rangle, \ \ H = \langle     (1, 3)(2, 4)(5, 7)(6, 8) \rangle.$$
\end{example}

Going forward it will be useful to have necessary and sufficient conditions for a finite group $G$ to be of generalized $D_4$-type. 

\begin{lemma}\label{lem:D4Classification}
A finite group $G$ is of generalized $D_4$-type if and only if it has exponent dividing 4 and nilpotency class at most $2$.
\end{lemma}

\begin{proof}
The necessity of these two conditions follows from the fact that $D_4$ has both of these properties and they are preserved under taking subgroups, direct products and quotients. 

What is left to prove is that these two conditions are sufficient to conclude that $G$ is of generalized $D_4$-type. We prove this by showing that the free group of nilpotency class 2 and exponent 4 on $k$ generators is isomorphic to a subgroup of $D_4^n$ for some $n$. We will denote the free group of nilpotency class 2 and exponent 4 on $k$ generators by $G_k$. Proving this is enough since any finite group that has nilpotency class at most 2 and exponent 4 is a quotient of one of these groups. We will proceed by induction on the number of generators. 
\medskip

\noindent
{Base case}: The group $G_1 \simeq \ZZ/4\ZZ$ which is clearly of generalized $D_4$-type.

\medskip

\noindent
{Inductive Assumption}: Suppose that $k \geq 1$ and $G_k$ is isomorphic to a subgroup of $D_4^n$ for some $n$. Let $x_1,\dots  x_{k+1}$ be the generators of $G_{k+1}$ and let $H_\ell$ be the smallest subgroup of $G_{k+1}$ containing all the generators except $x_\ell$. Notice that since $G_{k+1}$ is nilpotent of degree 2, the commutator subgroup of $G_{k+1}$ is contained in its center. Combining this with the fact that $G_{k+1}$ is exponent 4, we get that every element in $G_{k+1}$ can be written uniquely in the form 
$$x_1^{a_1}\cdots x_{k+1}^{a_{k+1}}[x_1,x_2]^{b_1}[x_1,x_3]^{b_2}\cdots[x_k,x_{k+1}]^{b_m},$$
where $[x_i,x_j]$ is the commutator of $x_i$ and $x_j$, $a_i\in \{0,1,2,3\}$, and $b_j\in \{0,1\}$.

Since all the relations defining these groups are symmetric, for each $1\leq\ell\leq k+1$, $H_\ell\simeq G_k$ and so by assumption $H_\ell$ is isomorphic to a subgroup of $D_4^n$. Let $\varphi_\ell\colon H_\ell \to D_4^n$ be an injective homomorphism, so that $H_\ell \simeq \Im( \varphi_\ell ) \subseteq D_4^n$. We extend each $\varphi_\ell$ to a homomorphism on all of $G_{k+1}$ by declaring $\varphi_\ell(x_\ell)$ to be trivial.

Using these we define a map $\psi\colon G_{k+1} \to (D_4^n)^{k+1}$ by $\psi(x) = \big(\varphi_1(x),\varphi_2(x), \dots, \varphi_{k+1}(x)\big)$. Clearly this is a homomorphism since each $\varphi_\ell$ is a homomorphism. The map $\psi$ is injective since $\ker(\varphi_\ell) = \langle x_\ell, [x_\ell,x_1],\dots, [x_\ell,x_{k+1}]\rangle$ and so the only element in $G_{k+1}$ that gets sent to the identity by every $\varphi_\ell$ is the identity element. 
\end{proof}

\begin{example}\label{ex:inducitonD4}
The proof of Lemma \ref{lem:D4Classification} is constructive and so we can compute explicitly the isomorphisms in each case. The group $G_2$ has the following presentation:
\begin{align*}
G_2 &= \langle x_1,x_2\mid x_1^4=x_2^4=(x_1x_2)^4 =[x_1,x_2]^2= e, [x_1,[x_1,x_2]]=[x_2,[x_1,x_2]]=e\rangle\\ 
\end{align*}
and is \href{https://groupprops.subwiki.org/wiki/SmallGroup(32,2)}{SmallGroup(32,2)} in Magma. Let $H = D_4^3$ and let $\tau_i$ and $\sigma_i$ be the elements of order 4 and 2 respectively that generate the $i$-th copy of $D_4$. It is an easy check to show that $G_2 \simeq \langle \tau_1\tau_2,\sigma_1\tau_3\rangle.$ AIf we let $G_{3}$ be generated by $x_1,x_2,$ and $x_3$, then 
\begin{align*}
H_1&= \langle x_2,x_3 \rangle \simeq \langle \tau_1\tau_2,\sigma_1\tau_3 \rangle,\\
H_2&= \langle x_3, x_1 \rangle \simeq \langle \tau_4\tau_5,\sigma_4\tau_6 \rangle,\\
H_3&= \langle x_1,x_2 \rangle \simeq \langle \tau_7\tau_8,\sigma_7\tau_9 \rangle.
\end{align*}
Again, it is a simple check in Magma verifies that
$$G_3 \simeq \big\langle \sigma_4\tau_6\tau_7\tau_8, \tau_1\tau_2\sigma_7\tau_9, \sigma_1\tau_3\tau_4\tau_5 \big\rangle.$$
\end{example}

\begin{remark}
The embedding of $G_3$ into $D_4^9$ is far from optimal. In fact, $G_3$ is a subgroup of a much smaller power of $D_4$. The smallest possible power it embeds into is 6 with
$$G_3\simeq \big\langle \tau_1\tau_2\tau_3,\sigma_2\sigma_3\tau_4\tau_5,\sigma_2\sigma_4\tau_5\tau_6\big\rangle.$$
\end{remark}

\begin{example}
Clearly, $\ZZ/8\ZZ$ is not of generalized $D_4$-type since its exponent is 8. On the other hand \href{https://groupprops.subwiki.org/wiki/Faithful_semidirect_product_of_E8_and_Z4}{SmallGroup(32,6)}, the faithful semidirect product of $(\ZZ/2\ZZ)^3$ and $\ZZ/4\ZZ$, has exponent 4 but nilpotency class 3 and so is not of generalized $D_4$-type. 
\end{example}

Let $K/\QQ$ be a number field. We will denote the Galois closure of $K$ by $\widehat{K}$.

\begin{remark}\label{remark:GaloisExtension}
It is clear from basic Galois theory that the compositum of any two fields of generalized $D_4$-type, is also of generalized $D_4$-type and any subfield of a field of generalized $D_4$-type will also be of generalized $D_4$-type. Since every finite extension $K\subseteq \QQ(D_4^\infty)$ over $\QQ$ is contained in a finite Galois extension $\widehat{K}/\QQ$ of generalized $D_4$-type, we can view $\QQ(D_4^\infty)$ as the compositum of Galois extensions and thus the infinite extension $\QQ(D_4^\infty)/\QQ$ is also Galois.
\end{remark}

\begin{lemma}\label{lem:QuadraticSubfield}
Given any square-free $d\in \ZZ$,  $\QQ(\sqrt{d})\subseteq\QQ(D_4^\infty)$.
\end{lemma}

\begin{proof}
For any such $d$, $\Gal(\QQ(\sqrt{d})/\QQ)\simeq\ZZ/2\ZZ$ which is of generalized $D_4$-type. 
\end{proof}

Going forward it will also be important to know which roots of unity are contained in $\QQ(D_4^\infty)$. Fortunately, using Lemma \ref{lem:D4Classification} classifying these subfields is fairly simple. 

\begin{lemma}\label{lem:CyclotomicSubfields}
Let $n$ be a natural number. Then,  $\QQ(\zeta_n)\subseteq\QQ(D_4^\infty)$ if and only if $n$ divides 240.
\end{lemma}

\begin{proof}
Let $\lambda(n)$ be the exponent of $\Gal(\zeta_n/\QQ)\simeq(\ZZ/n\ZZ)^\times$. It is a classical result that $\lambda(2^e) = 2^{e-2}$ and $\lambda(p^e) = (p-1)p^{e-1}$ for primes $p>2$. It follows that if $\ell$ is prime then $\lambda(\ell^j)$ divides 4  if and only $\ell^j\in\{2, 3, 4, 5, 8, 16 \}$. In these cases $\Gal(\QQ(\zeta_n)/\QQ)\simeq (\ZZ/n\ZZ)^\times$ which has nilpotency class 1 and so Lemma \ref{lem:D4Classification} shows that $\QQ(\zeta_n)\subseteq\QQ(D_4^\infty).$
\end{proof}

\section{Growth of the torsion subgroups of elliptic curves by base extension}\label{sec:GrowthOfTorsion}

In this section, we present results about the fields of definition of torsion subgroups of elliptic curves that will be useful throughout the rest of the paper. 

\begin{prop}\cite[Ch. III, Cor. 8.1.1]{silverman}\label{prop:ContiansCyclo}
Let $E/L$ be an elliptic curve with $L\subseteq \overline{\QQ}.$ For each integer $n\geq 1$, if $E[n]\subseteq E(L)$ then the $n$th cyclotomic field $\QQ(\zeta_n)$ is a subfield of $L$.
\end{prop}

\begin{theorem}\label{thm:finiteness}\cite[Theorem 4.1]{Q(3)}
Let $E/\QQ$ be an elliptic curve and let $F$ be a (possibly infinite) Galois extension of $\QQ$ that only contains finitely many roots of unity. Then $E(F)_\tor$ is finite. Moreover, there is a uniform bound $B$, depending only on $F$, such that $\#E(F)_\tor\leq B$ for every elliptic curve $E/\QQ$.
\end{theorem}

Theorem \ref{thm:finiteness} together with Lemma \ref{lem:CyclotomicSubfields} shows that if $E/\QQ$ is an elliptic curve, then $E(\QQ(D_4^\infty))_\tor$ must be finite and immediately limits the possible $n$'s such that $E[n]\subseteq E(\QQ(D_4^\infty))$. Below we present the main reults used to prove Theorem \ref{thm:finiteness} as they will be useful in the next few sections.

\begin{lemma}\cite[Lemma 4.6]{Q(3)}\label{lem:isogp^j-k}
Let $E$ and $F$ be as in Theorem \ref{thm:finiteness}, let $p$ be a prime, and let $k$ be the largest integer for which $E[p^k]\subseteq E(F)$. If $E(F)_\tor$ contains a subgroup isomorphic to $\ZZ/p^k\ZZ\oplus \ZZ/p^j\ZZ$ with $j\geq k$, then $E$ admits a rational $p^{j-k}$-isogeny.
\end{lemma}

\begin{theorem}[\cite{mazur2},\cite{kenku2},\cite{kenku3},\cite{kenku4},\cite{kenku5}]\label{thm:IsoTypes}
Let $E/\QQ$ be an elliptic curve with a rational $n$-isogeny. Then 
\[
n\leq 19\hbox{ or }n\in\{21,25,27,37,43,67,163\}.
\]
\end{theorem}

With these results, it is clear that it will be helpful to better understand the field of definition of the kernel a given isogeny.  

\begin{lemma}\cite[Lemma 4.8]{Q(3)}\label{lem:CyclicIso}
Let $E/\QQ$ be an elliptic curve that admits a rational $n$-isogeny $\varphi$, and let $R\in E[n]$ be a point of order $n$ in the kernel of $\varphi$. The field extension $\QQ(R)/\QQ$ is Galois and $\Gal(\QQ(R)/\QQ)$ is isomorphic to a subgroup of $(\ZZ/n\ZZ)^\times$. In particular, if $n$ is prime, then $\Gal(\QQ(R)/\QQ)$ is cyclic and its order divides $n-1$.
\end{lemma}

\begin{prop}\cite[Theorem 2.1]{lozano1}\label{prop:FieldOfDef}
Let $E/\QQ$ be an elliptic curve and let $p\geq 11$ be a prime, other than $13$. Let $R\in E[p]$ be a torsion point of exact order $p$ and let $\QQ(R) = \QQ(x(R),y(R))$ be the field of definition of $R$. Then
\[
[\QQ(R):\QQ] \geq \frac{p-1}{2}
\]
unless $j(E) = -7\cdot 11^3$ and $p=37$, in which case $[\QQ(R):\QQ]\geq (p-1)/3 = 12$.
\end{prop}

\section{The maximal $p$-primary components of \texorpdfstring{$E(\QQ(D_4^\infty))_\tor$}{D4}}\label{sec:pPrimary}

We are now ready to start to classifying the groups $E(\QQ(D_4^\infty))_\tor$ up to isomorphism as $E$ ranges over all elliptic curves defined over $\QQ$. The first step is to compute a bound on $\#E(\QQ(D_4^\infty))_\tor$ by obtaining bounds on the $p$-primary components of $E(\QQ(D_4^\infty))$ for elliptic curves $E$ defined over $\QQ$ \emph{without} complex multiplication. 
\begin{theorem}\label{thm:D4-upperbound}
Let $E/\QQ$ be an elliptic curve without complex multiplication. Then $E(\QQ(D_4^\infty))_\tor$ is isomorphic to a subgroup of
\[
T_{\rm max} = (\ZZ/16\ZZ\oplus\ZZ/32\ZZ)\oplus(\ZZ/3\ZZ\oplus\ZZ/9\ZZ)\oplus(\ZZ/5\ZZ\oplus\ZZ/5\ZZ)\oplus\ZZ/7\ZZ\oplus\ZZ/13\ZZ,
\]
and $T_{\rm max}$ is the smallest group with this property.\end{theorem}

To prove Theorem \ref{thm:D4-upperbound} we first notice that there are only finitely many primes such that the $p$-primary component of $E(\QQ(D_4^\infty))_\tor$ could be nontrivial.

\begin{prop}\label{prop:FinitePrimesD4}
Let $E/\QQ$ be an elliptic curve, and let $p$ be a prime dividing the cardinality of $E(\QQ(D_4^\infty))_\tor$. Then $p\in \{ 2, 3, 5, 7, 13\}$.
\end{prop}

\begin{proof}
From Proposition \ref{prop:ContiansCyclo} and Lemma \ref{lem:isogp^j-k} the $p$-primary component of $E(\QQ(D_4^\infty))$ maybe nontrivial if and only if $\QQ(\zeta_p)\subseteq \QQ(D_4^\infty)$ or $E$ has a rational $p$-isogeny. With this, Lemma \ref{lem:CyclotomicSubfields}, and Theorem \ref{thm:IsoTypes} the only primes that can divide the cardinality of $E(\QQ(D_4^\infty))$ are exactly the ones in $S = \{ 2,3,5, 7,11, 13, 17,19,37,43,67,163 \}$. From Lemma \ref{lem:CyclicIso} if $E(\QQ(D_4^\infty))[p] = \langle R \rangle \simeq \ZZ/p\ZZ$, then $\QQ(R)/\QQ$ is a cyclic extension and of generalized $D_4$-type. Therefore the extension is at most degree 4. Combining this with Proposition \ref{prop:FieldOfDef} completes the proof. 
\end{proof}

\begin{remark}\label{remark:FinitePrimesD4WorksForCM}
Proposition \ref{prop:FinitePrimesD4} does not require that $E/\QQ$ be an elliptic curve without complex multiplication. This observation will be useful when we are dealing with the case of elliptic curves with complex multiplication in Section \ref{subsec:CMCurves}. 
\end{remark}

Recall that the $\Qbar$-isomorphism class of an elliptic curve $E/\QQ$ may be identified with its $j$-invariant $j(E)$.

\begin{proposition}\label{prop:twist}
Let $E/\QQ$ be an elliptic curve with $j(E)\ne 0$.
The isomorphism type of $E(\QQ(D_4^\infty))_\tor$ depends only on the $\Qbar$-isomorphism class of $E$, equivalently, only on $j(E)$.
\end{proposition}
\begin{proof}
Recall that for $j(E)\ne 0,1728$, if $j(E')=j(E)$ for some $E'/\QQ$ then $E'$, is a quadratic twist of $E$, hence isomorphic to $E$ over an extension of degree at most 2.
If $j(E)=1728 = j(E')$, then $E'/\QQ$ is isomorphic to $E$ over a field of the form $\QQ(\sqrt[4]{n})$ for some $n\in\ZZ$ \cite[\S X.5]{silverman}. The Galois closure of such a field is isomorphic to a subgroup of $D_4$ and so for $j(E)=j(E')\ne 0$, the elliptic curves $E$ and $E'$ are isomorphic over a field of generalized $D_4$-type, hence their base changes to $\QQ(D_4^\infty)$ are isomorphic and 
$E(\QQ(D_4^\infty))_\tor \simeq E'(\QQ(D_4^\infty))_\tor$.
\end{proof}

\begin{lemma}\label{lem:j=0}
There are three possible torsion structures over $\QQ(D_4^\infty)$ given an elliptic curve $E/\QQ$ with $j(E) = 0$. They are realized by the curves \hbox{\rm \href{http://www.lmfdb.org/EllipticCurve/Q/27a1}{\texttt{27a1}}, \href{http://www.lmfdb.org/EllipticCurve/Q/36a1}{\texttt{36a1}},} and \hbox{\rm \href{http://www.lmfdb.org/EllipticCurve/Q/108a1}{\texttt{108a1}}.}
\end{lemma}

\begin{proof}
Every elliptic curve $E/\QQ$ with $j(E)=0$ is isomorphic over $\QQ$ to a curve of the form
$$E_s: y^2=x^3 + s$$ 
for some $s\in \ZZ\setminus\{0\}$ that is 6th power free. The three division polynomial of $E_s$ is given by $f_3(x) =  3x(x^3+4s)$ and so we can see that generically these curve will have a point of order 3 defined over a quadratic field and hence $\QQ(D_3^\infty)$.
Further inspection shows that if $4s=t^3$ for some $t\in\QQ$ then $E[3]\subseteq E(\QQ(D_4^\infty))$ since this would mean that $E$ would have it's full 3-torsion defined over a 2-elementary extension of $\QQ$. In this case $\Im\bar\rho_{_s,3}$ is contained in a group conjugate to a subgroup of the split Cartan subgroup of $\GL_2(\ZZ/3\ZZ)$. In fact, if $4s$ is a cube, the factorization of the 9-division polynomial of $E_s$ shows shows that $E_s$ has a 3-isogeny and a 9-isogeny that are independent of each other. Notice also that if $t,r\in\ZZ\setminus\{0\}$, then the curves $E_{2t^3}$ and $E_{2r^3}$ are isomorphic over $\QQ(\sqrt{rt})\subseteq\QQ(D_4^\infty)$. Therefore, all over these curves have the same torsion subgroup over $\QQ(D_4^\infty)$ as $E_2(\QQ(D_4^\infty))_\tor\simeq \ZZ/3\ZZ\oplus\ZZ/3\ZZ$.

From \cite[Table 3 \& 4]{lozano1}, the only other possible isogeny type that $E_s$ can have is a 2-isogeny. This occurs exactly when $s = t^3$ for some $t\in \ZZ$ and thus has a point of order 2 defined of $\QQ$ and in fact $E[4]\subseteq E(\QQ(D_4^\infty))$. Again, if $t,r\in\ZZ\setminus\{0\}$, then $E_{t^3}$ is isomorphic to $E_{r^3}$ over $\QQ(\sqrt{rt})\subseteq \QQ(D_4^\infty)$ and so $E_{t^3}(\QQ(D_4))_\tor \simeq E_{r^3}(\QQ(D_4^\infty))_\tor$. Thus for every $r\in\ZZ\setminus\{0\}$, $E_{r^3}(D_4^\infty)_\tor\simeq E_{1}(\QQ(D_4^\infty))_\tor\simeq \ZZ/8\ZZ\oplus\ZZ/24\ZZ.$\footnote{The curve $E_{t^3}$ has its full 8-torsion defined over $\QQ(D_4^\infty)$ because it has discriminant $-432t^6$ and its 2-isogenous curve has discriminant $6912t^6$ and so $-432t^6 \equiv (-1) (6912t^6)\bmod(\QQ^\times)^2$. This is enough to show that $\Im\bar\rho_{E_{t^3},8}$ is of generalized $D_4$-type.}

The analysis in the sections following this lemma we show that for every other prime $p$, the $p$-primary componenet of $E_s(\QQ(D_4^\infty))$ must be trivial. Bringing this all together we have that
$$E_s(\QQ(D_4^\infty))_\tor\simeq \begin{cases}
\ZZ/3\ZZ\oplus\ZZ/3\ZZ & \hbox{ if 4s is a cube},\\
\ZZ/8\ZZ\oplus\ZZ/24\ZZ& \hbox{ if s is a cube},\\
\ZZ/3\ZZ & \hbox{ otherwise}.
\end{cases}
$$
Each of these three cases are realized by the three curves listed in the statement of the Lemma.
\end{proof}

\subsection{The case when $p=13$.}\label{subsec:Tmax13}

\begin{prop}\label{prop:D4Classification13}
Suppose that $E/\QQ$ is an elliptic curve such that $13$ divides $\#E(\QQ(D_4^\infty))_\tor.$ Then, $E(\QQ(D_4^\infty))(13)\simeq\ZZ/13\ZZ$ and there exists $t\in\QQ$ such that 
\[
j(E) = \frac{(t^2-t+1)^3P(t)^3}{(t-1)^{13}t^{13}(t^3-4t^2+t+1)}
\]
where $P(t) = t^{12}-9t^{11}+29t^{10}-40t^9+22t^8-16t^7+40t^6-22t^5-23t^4+25t^3-4t^2-3t+1.$
\end{prop}

\begin{proof}
Since $\QQ(\zeta_{13})\not\subseteq \QQ(D_4^\infty)$ it must be that $E(\QQ(D_4^\infty))[13] \simeq \ZZ/13\ZZ$. Suppose that $R$ is a point that generates $E(\QQ(D_4^\infty))[13]$. By Lemma \ref{lem:isogp^j-k} and Proposition \ref{prop:FieldOfDef}, $\QQ(R)/\QQ$ must be a cyclic extension of degree dividing $12$. Further, since this point is defined over $\QQ(D_4^\infty)$ from Lemma \ref{lem:D4Classification} the degree $\QQ(R)/\QQ$ must divide 4 and $\Im\bar\rho_{E,13}$ is conjugate to a subgroup contained inside of the matrices of the form $\begin{pmatrix} a^3 & * \\ 0 & * \end{pmatrix}$. 

The elliptic curves defined over $\QQ$ with this property have been completely classified in \cite{zywina1}, and they correspond to the curves with $j$-invariant of the form in the statement of the proposition. Further, since there are no elliptic curves defined over $\QQ$ with a cyclic 169-isogeny, it is not possible for $E(\QQ(D_4^\infty))(13)$ to be any larger. 
\end{proof}

\subsection{The case when $p=7$.}\label{subsec:Tmax7}

\begin{prop}\label{prop:D4Classification7}
Suppose that $E/\QQ$ is an elliptic curve such that $7$ divides $\#E(\QQ(D_4^\infty))_\tor.$ Then, $E(\QQ(D_4^\infty))(7)\simeq\ZZ/7\ZZ$ and there exists $t\in\QQ$ such that 
\[
j(E) = \frac{(t^2-t+1)^3(t^6-11t^5+30t^4-15t^3-10t^2+5t+1)^3}{(t-1)^7t^7(t^3-8t^2+5t+1)}.
\]

\end{prop}

\begin{proof}
Again, since $\QQ(\zeta_{7})\not\subseteq \QQ(D_4^\infty)$ and there are no elliptic curves defined over $\QQ$ with a 49-isogeny it must be that $E(\QQ(D_4^\infty))(7) \simeq \ZZ/7\ZZ$. Suppose that $R$ is a point that generates $E(\QQ(D_4^\infty))(7)$. This time, the extension $\QQ(R)/\QQ$ is cyclic of degree dividing 6. Therefore, if $\QQ(R)\subseteq\QQ(D_4^\infty)$ it must be degree 2 or 1. If $[\QQ(R):\QQ]$ is in fact degree 2, this means that $E$ has a point of order 7 defined over a quadratic extension of $\QQ$. In fact, if $E$ has a point of order $7$ defined over a quadratic field, then there is a quadratic twist of $E$ that has a order of order 7 defined over $\QQ$. These curves are parameterized by a genus 0 modular curve and the $j$-map can again be found in \cite{zywina1}. 
\end{proof}

\subsection{The case when $p=5$.}\label{subsec:Tmax5} 

\begin{prop}\label{prop:D4Classification5}
Suppose that $E/\QQ$ is an elliptic curve such that $5$ divides $\#E(\QQ(D_4^\infty))_\tor.$ Then it must be that $E(\QQ(D_4^\infty))(5)\simeq\ZZ/5\ZZ$ or $\ZZ/5\ZZ\oplus\ZZ/5\ZZ$. Further, $E(\QQ(D_4^\infty))(5)$ has a subgroup isomorphic to $\ZZ/5\ZZ$ exactly when there exists a $t\in \QQ$ such that 
\[
j(E) = \frac{5^2(t^2+10t+5)^3}{t^5},
\]
while $E(\QQ(D_4^\infty))(5)\simeq\ZZ/5\ZZ\oplus \ZZ/5\ZZ$ exactly when there exists a $t\in \QQ$ such that 
\[
j(E) = \frac{(t^2+5t+5)^3(t^4+5t^2+25)^3(t^4+5t^3+20t^2+25t+25)^3}{t^5(t^4+5t^3+15t^2+25t+25)^5}.
\]
\end{prop}

\begin{proof}
If $E(\QQ(D_4^\infty))(5)\simeq\ZZ/5\ZZ$, the single point of order 5 must be defined over a cyclic extension of order dividing 4 by Lemma \ref{lem:CyclicIso}. All of these extensions are of generalized $D_4$-type, which is to say that having a $5$-isogeny is necessary and sufficient for $E$ to have a point of order 5 defined over $\QQ(D_4^\infty)$. Elliptic curves defined over $\QQ$ with a rational 5-isogeny are parametrized by the genus 0 modular curve $X_0(5)$ and the $j$-map can again be found in \cite{zywina1} and is listed above.

If $E(\QQ(D_4^\infty))(5)\simeq\ZZ/5\ZZ\oplus\ZZ/5\ZZ$ then $\Gal(\QQ(E[5])/\QQ) \simeq \Im\bar\rho_{E,5}\subseteq \GL_2(\ZZ/5\ZZ)$ must be of generalized $D_4$-type. Searching for subgroups of $\GL_2(\ZZ/5\ZZ)$ up to conjugation that are of generalized $D_4$-type, have surjective determinant, and have an element of trace 0 and determinant -1, we find that up to conjugation they are all contained in a single maximal group. More precisely, if $\Im\bar\rho_{E,5}$ is of generalized $D_4$-type, then it is conjugate to a subgroup of the split-Cartan subgroup of $\GL_2(\ZZ/5\ZZ)$. Elliptic curves with this property are again parametrized by a genus 0 modular curve whose $j$-map is given in \cite{zywina1}.

To complete the proof, all that is left to do is prove that $E(\QQ(D_4^\infty))(5)$ cannot contain a point of order 25. Suppose towards a contradiction that $E(\QQ(D_4^\infty))$ contains a point of order 25. There are two ways that this could happen: either $E(\QQ(D_4^\infty))(5) \simeq \ZZ/25\ZZ$ or it contains a group isomorphic to $\ZZ/5\ZZ\oplus\ZZ/25\ZZ$. 

In the first case, from Lemmas \ref{lem:isogp^j-k} and \ref{lem:CyclicIso} together with our classification of groups of generalized $D_4$-type, the first case can only occur if $E$ has a point of order 25-defined over a cyclic quartic extension of $\QQ$, but from \cite[Theorem 1.2]{chou1} this can't happen.

Next, suppose that $E(\QQ(D_4^\infty)$ contains a subgroup isomorphic to $\ZZ/5\ZZ \oplus \ZZ/25\ZZ$. This would mean that $G = \Im(\bar\rho_{E,25})\subseteq \GL_2(\ZZ/25\ZZ)$ has the following two properties:
\begin{enumerate}
\item $G$ has a surjective determinant map and an element with trace 0 and determinant $-1$,
\item $G$ contains a normal subgroup $N$ that acts trivially on a $\ZZ/25\ZZ$-submodule of $\ZZ/25\ZZ\oplus\ZZ/25\ZZ$ isomorphic to $\ZZ/5\ZZ\oplus\ZZ/25\ZZ$ for which $G/N$ is of generalized $D_4$-type.
\end{enumerate}
The first property is comes from the discussion in Section \ref{sec:galoisreps}, while the second property reflects the requirement that $\QQ(E[25])$ contains the Galois extension $\QQ(E(\QQ(D_4^\infty))[25])/\QQ$ whose Galois group is isomorphic to $G/N$ and for which the Galois group $\Gal(\QQ(E[27])/\QQ( E(\QQ(D_4^\infty))[25]))\simeq N$ acts trivially on a subgroup of $E[25]$ which is isomorphic to $\ZZ/5\ZZ\oplus\ZZ/25\ZZ$. 

Enumerating such groups in Magma, we find that there are 3 maximal groups with properties (1) and (2) and each of them is conjugate to a subgroup of 
\[
H = \left\langle 
\begin{pmatrix} 7 & 0 \\ 0 & 1 \end{pmatrix}, \begin{pmatrix} 1 & 0 \\ 0 & 2 \end{pmatrix},  \begin{pmatrix} 1 & 1 \\ 0 & 1 \end{pmatrix}
\right\rangle.
\]
Inspecting this group, we see that $\Im\bar\rho_{E,25}$ can only be conjugate to a subgroup $H$ if $E$ has a point of order 25 defined over a cyclic quartic field and again this is impossible from \cite{chou1}.
\end{proof}

\begin{cor}\label{cor:Nontriv5ImplesIsog}
Suppose that $E/\QQ$ is an elliptic curve without complex multiplication such that $5$ divides $\#E(\QQ(D_4^\infty))_\tor.$ Then $E$ admits a rational $5$-isogeny. In particular if $E(\QQ(D_4^\infty))(5) \simeq \ZZ/5\ZZ\oplus\ZZ/5\ZZ$, then $E$ admits two independent 5-isogenies.
\end{cor}

\subsection{The case when $p = 3$.}\label{subsec:Tmax3} 

\begin{prop}\label{prop:D4Classification3}
Suppose that $E/\QQ$ is an elliptic curve without complex multiplication such that $3$ divides $\#E(\QQ(D_4^\infty))_\tor.$ Then it must be that $E(\QQ(D_4^\infty))(3)\simeq\ZZ/3\ZZ$, $\ZZ/9\ZZ$, or $\ZZ/3\ZZ\oplus\ZZ/3\ZZ$. Further, $E(\QQ(D_4^\infty))(3)$ has a subgroup isomorphic to $\ZZ/3\ZZ$ exactly when there exists a $t\in \QQ$ such that 
\[
j(E) = 27\frac{(t+1)(t+9)^3 }{t^3},
\]
$E(\QQ(D_4^\infty))(3)$ has a subgroup isomorphic to $\ZZ/9\ZZ$ exactly when there exists a $t\in \QQ$ such that 
\[ 
j(E) = \frac{(t^3 - 3t^2 + 1)^3  (t^9 - 9t^8 + 27t^7 - 48t^6 + 54t^5 - 45t^4 + 27t^3 - 9t^2 + 1)^3}{(t-1)^9t^9(t^2-t+1)^2(t^3-6t^2+3t+1) },
\]
$E(\QQ(D_4^\infty))(3)$ has a subgroup isomorphic to $\ZZ/3\ZZ\oplus\ZZ/3\ZZ$ exactly when there exists a $t\in \QQ$ such that 
\[
j(E) = 27\frac{(t+1)^3(t-3)^3 }{t^3}.
\]
\end{prop}

\begin{proof}
Suppose that $E(\QQ(D_4^\infty))(3)\simeq\ZZ/3^k\ZZ$. In this case, in order for these points to be defined over $\QQ(D_4^\infty)$ it must be that they are defined over a quadratic field since $(\ZZ/3^k\ZZ)^\times$ is cyclic of order $2\cdot 3^{k-1}$. Thus $E$ must have a quadratic twist with a rational point point of order $3^k$ and from Theorem \ref{thm:mazur} this is only possible when $k=1$ or 2. The curves with this property are again parameterized by genus 0 modular curves $X_1(3)$ and $X_1(9)$ and their $j$-maps can be found in many places including \cite{SZ}. 

Next assume that $\QQ(E[3])\subseteq \QQ(D_4^\infty)$. Searching for subgroups of $\GL_2(\ZZ/3\ZZ)$ up to conjugation that are of generalized $D_4$-type, have surjective determinant, and have an element of trace 0 and determinant $-1$, we find that all of these groups are contained inside the normalizer of the split-Cartan subgroup of $\GL_2(\ZZ/3\ZZ)$. These curves are parameterized by a genus zero modular curve and the $j$-map can be found in \cite{zywina1}.

The last case we need to consider is if it is possible for $E(\QQ(D_4^\infty))_\tor$ to contain a subgroup isomorphic to $\ZZ/3\ZZ\oplus \ZZ/9\ZZ.$ Just as we did in the proof of Proposition \ref{prop:D4Classification5}, suppose towards a contradiction that $E$ is such a curve and let $G = \Im\bar\rho_{E,9}\subseteq\GL_2(\ZZ/9\ZZ)$.  Using Magma we search for subgroups $G$ of $\GL_2(\ZZ/9\ZZ)$ up to conjugation such that 
\begin{enumerate}
\item $G$ has a surjective determinant map and an element with trace 0 and determinant $-1$,
\item $G$ contains a normal subgroup $N$ that acts trivially on a $\ZZ/9\ZZ$-submodule of $\ZZ/9\ZZ\oplus\ZZ/9\ZZ$ isomorphic to $\ZZ/3\ZZ\oplus\ZZ/9\ZZ$ for which $G/N$ is of generalized $D_4$-type.
\end{enumerate}
Again, there is exactly one maximal group $H$ with this property,
\[
H = \left\langle
\begin{pmatrix} 1 & 3 \\ 0 & 1 \end{pmatrix}, \begin{pmatrix} 1 & 0 \\ 0 & 2 \end{pmatrix}, \begin{pmatrix} 8 & 0 \\ 0 & 8 \end{pmatrix}
\right\rangle.
\]
Inspecting $H$ we see that in order for $\Im\bar\rho_{E,9}\subseteq H$, $E$ would have to have a quadratic twist with a point of order 9 and another independent 3-isogeny. Since there is only one $\Qbar$-isomorphism class of elliptic curves with a 27-isogeny (see for example \cite[Table 4]{lozano1}), there is only one elliptic curve up to $\Qbar$-isomorphism with independent 3- and 9-isogenies. This is the class of CM elliptic curves with $j$-invariant equal to $0$. 
\end{proof}

\begin{remark}
Unlike the case when $p=5$, there are elliptic curves $E/\QQ$ such that $3$ divides $\#E(\QQ(D_4^\infty))_\tor$, but $E$ does not admit a rational $3$-isogeny. 
\end{remark}

\begin{example}
Let $E/\QQ$ be the elliptic curve with Cremona label \href{http://www.lmfdb.org/EllipticCurve/Q/338e1}{\texttt{338e1}}. Then $\Im \bar\rho_{E,3}$ is exactly the normalizer of the split Cartan subgroup of $\GL_2(\ZZ/3\ZZ)$. This curve has the property that $E(\QQ(D_4^\infty))(3) \simeq \ZZ/3\ZZ\oplus\ZZ/3\ZZ$, but $E$ does not admit a rational 3-isogeny. 
\end{example}

\subsection{The case when $p=2$.}\label{subsec:Tmax2}

\begin{prop}\label{prop:Triv2TorOrFull4Tor}
Given an elliptic curve $E/\QQ$, either $E(\QQ(D_4^\infty))[2]$ is trivial or $E[4] \subseteq E(\QQ(D_4^\infty))$. 
\end{prop}

\begin{proof}
Notice that generically one expects that $\Gal(\QQ(E[2])/\QQ)\simeq \GL_2(\ZZ/2\ZZ) \simeq S_3$ which is not of generalized $D_4$-type. Further, in order for $E(\QQ(D_4^\infty))[2]$ to be nontrivial it must be that $\QQ(E[2])/\QQ$ is either trivial or a quadratic extension. This is equivalent to $\Im\bar\rho_{E,2}$ being trivial or conjugate to $G = \left\langle \left(\begin{smallmatrix} 1&1\\0&1 \end{smallmatrix}\right)\right\rangle $. A simple check shows that the preimage of these groups under the standard component-wise reduction map $\pi\colon \GL_2(\ZZ/4\ZZ)\to\GL_2(\ZZ/2\ZZ)$ is of generalized $D_4$-type. Thus, $E(\QQ(D_4^\infty))[2]$ is nontrivial exactly when $E$ has at least one point of order 2 defined over $\QQ$ and in this case it must be that case $E[4]\subseteq E(\QQ(D_4^\infty))$. 
\end{proof}

There are many additional issues when dealing with the case of $p=2$ that do not arise in the previous cases. In particular, since there are rational elliptic curves with $16$-isogenies and $\QQ(\zeta_{16})\subseteq\QQ(D_4^\infty)$ the potential size of $E(\QQ(D_4^\infty))(2)$ is much greater than the previous primes. While this does make things more difficult, we are fortunate that all the possible images of the 2-adic representations associated to rational elliptic curves have all been completely classified by Rouse and Zureick-Brown in \cite{RZB}. Rather than ignoring these results and trudging through the many cases, we search the data available at their website using Magma and classify all the possible 2-torsion structures that can occur over $\QQ(D_4^\infty)$. Doing this yields the following proposition.

\begin{prop}\label{prop:D4Classification2}
Suppose that $E/\QQ$ is an elliptic curve without complex multiplication such that $2$ divides $\#E(\QQ(D_4^\infty))_\tor.$ Then it must be that $E(\QQ(D_4^\infty))(2)$ is isomorphic to one of the following group:
\[
\ZZ/4\ZZ\oplus\ZZ/4\ZZ,\ \ZZ/4\ZZ\oplus\ZZ/8\ZZ,\ \ZZ/4\ZZ\oplus\ZZ/16\ZZ,\ \ZZ/4\ZZ\oplus\ZZ/32\ZZ,
\]
\[ \ZZ/8\ZZ\oplus\ZZ/8\ZZ,\ \ZZ/8\ZZ\oplus\ZZ/16\ZZ,\ \ZZ/8\ZZ\oplus\ZZ/32\ZZ, 
\hbox{ or } \ZZ/16\ZZ\oplus\ZZ/16\ZZ.
\]
Further, $E(\QQ(D_4^\infty))(2)$ contains a subgroup isomorphic to each these options if and only if they come from a rational points on the genus $0$ curves given in Table \ref{tab:RZBData}.
\end{prop}

For the sake of brevity, we leave the $j$-maps for each of these genus 0 curve out of the statement of this Proposition \ref{prop:D4Classification2} and this section. Instead, in Table \ref{tab:RZBData} we list each possible nontrivial 2-primary component and the corresponding modular curves (in the notation of \cite{RZB}) that parametrize the curves with these $\QQ$. We also include Figure \ref{fig:RZB}, that shows how these curves are related to each other. There is a line between two curves if the top one is covered by the bottom.

\begin{table}
\centering
\renewcommand{\arraystretch}{1.2}
\begin{tabular}{|l|l||l|l|}\hline
$E(\QQ(D_4^\infty))(2)$ & RZB Curve&$E(\QQ(D_4^\infty))(2)$ & RZB Curve\\\hline
$\ZZ/4\ZZ\times\ZZ/4\ZZ$ & \href{http://users.wfu.edu/rouseja/2adic/X6.html}{$X_6$}  &
$\ZZ/8\ZZ\times\ZZ/8\ZZ$ &   \href{http://users.wfu.edu/rouseja/2adic/X8.html}{$X_{8}$}, \href{http://users.wfu.edu/rouseja/2adic/X11.html}{$X_{11}$}   \\
$\ZZ/4\ZZ\times\ZZ/8\ZZ$ &  \href{http://users.wfu.edu/rouseja/2adic/X12.html}{$X_{12}$}, \href{http://users.wfu.edu/rouseja/2adic/X13.html}{$X_{13}$}  &
$\ZZ/8\ZZ\times\ZZ/16\ZZ$ &  \href{http://users.wfu.edu/rouseja/2adic/X25.html}{$X_{25}$}, \href{http://users.wfu.edu/rouseja/2adic/X92.html}{$X_{92}$}   \\
$\ZZ/4\ZZ\times\ZZ/16\ZZ$ &  \href{http://users.wfu.edu/rouseja/2adic/X36.html}{$X_{36}$}  &
$\ZZ/8\ZZ\times\ZZ/32\ZZ$ &  \href{http://users.wfu.edu/rouseja/2adic/X193.html}{$X_{193}$}   \\
$\ZZ/4\ZZ\times\ZZ/32\ZZ$ &  \href{http://users.wfu.edu/rouseja/2adic/X235.html}{$X_{235}$} &
$\ZZ/16\ZZ\times\ZZ/16\ZZ$ & \href{http://users.wfu.edu/rouseja/2adic/X58.html}{$X_{58}$}      \\\hline
\end{tabular}
\caption{Parameterization of the possible nontrivial 2-primary components}  \label{tab:RZBData}
\end{table}
\renewcommand{\arraystretch}{1}

\begin{figure}
$$\xymatrix{
 X_{58}\ar@{-}[rd]  & X_{193}\ar@{-}[d]  &   &   & X_{192}\ar@{-}[dd]\ar@{-}[dr]  &   & X_{235}\ar@{-}[ld]  \\\
   & X_{25}\ar@{-}[d]  &   &   &   & X_{36}\ar@{-}[d]  &   \\\
   & X_8 \ar@{-}[rrd] & X_{12}\ar@{-}[rd]  &   & X_{11}\ar@{-}[ld]  & X_{13}\ar@{-}[lld]  &   \\\
   &   &   & X_6  &   &   &   \\\
}$$
\caption{Covering relationships between the curves in Table \ref{tab:RZBData}}\label{fig:RZB}
\end{figure}
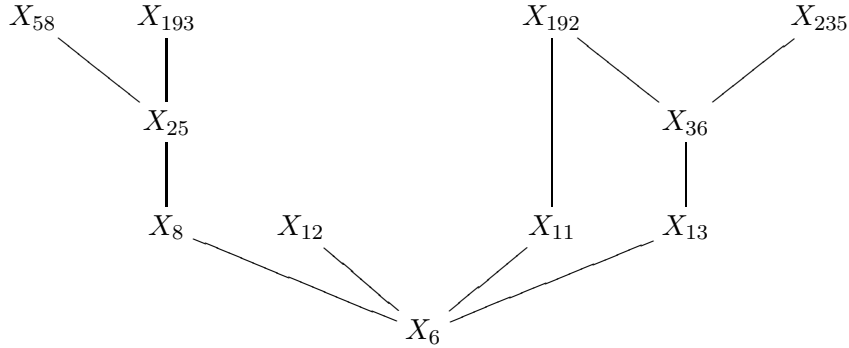

Together Propositions \ref{prop:D4Classification13}, \ref{prop:D4Classification7}, \ref{prop:D4Classification5}, \ref{prop:D4Classification3}, and \ref{prop:D4Classification2} complete the proof of Theorem \ref{thm:D4-upperbound}. 

\subsection{The case when \texorpdfstring{$E$}{E} has complex multiplication.}\label{subsec:CMCurves}

As mentioned in Proposition \ref{prop:twist} and Lemma \ref{lem:j=0}, if $E/\QQ$ is an elliptic curve, as long as $j(E) \neq 0$, the isomorphism class of $E(\QQ(D_4^\infty))_\tor$ does not change under twisting and if $j(E)$ is 0, then it is sufficient to compute the torsion structure for the curves \href{http://www.lmfdb.org/EllipticCurve/Q/27a1}{\texttt{27a1}}, \href{http://www.lmfdb.org/EllipticCurve/Q/36a1}{\texttt{36a1}}, and \href{http://www.lmfdb.org/EllipticCurve/Q/108a1}{\texttt{108a1}}. Using Proposition \ref{prop:FinitePrimesD4} to limit the computations, we compute $E(\QQ(D_4^\infty))_\tor$ for each of the necessary representatives. The results of these computations can be found in Table \ref{tab:CMTor}.

\begin{table}
\centering
\renewcommand{\arraystretch}{1.2}
\begin{tabular}{|l|l|| l| l|}\hline
$E/\QQ$ & $E(\QQ(D_4^\infty))_\tor$&$E/\QQ$ & $E(\QQ(D_4^\infty))_\tor$\\\hline
\href{http://www.lmfdb.org/EllipticCurve/Q/27a1}{\texttt{27a1}} & $\ZZ/3\ZZ \oplus \ZZ/3\ZZ$ &
\href{http://www.lmfdb.org/EllipticCurve/Q/108a2}{\texttt{108a2}} & $\ZZ/3\ZZ$  \\
\href{http://www.lmfdb.org/EllipticCurve/Q/27a4}{\texttt{27a4}} & $\ZZ/3\ZZ$ &
\href{http://www.lmfdb.org/EllipticCurve/Q/121b1}{\texttt{121b1}} & $\ZZ/3\ZZ \oplus \ZZ/3\ZZ$ \\
\href{http://www.lmfdb.org/EllipticCurve/Q/32a1}{\texttt{32a1}} & $\ZZ/8\ZZ \oplus \ZZ/8\ZZ$ &
\href{http://www.lmfdb.org/EllipticCurve/Q/256d1}{\texttt{256d1}} & $\ZZ/12\ZZ \oplus \ZZ/24\ZZ$  \\
\href{http://www.lmfdb.org/EllipticCurve/Q/32a4}{\texttt{32a4}} & $\ZZ/4\ZZ \oplus \ZZ/8\ZZ$ &
\href{http://www.lmfdb.org/EllipticCurve/Q/361a1}{\texttt{361a1}} & $\{\mathcal{O}\}$ \\
\href{http://www.lmfdb.org/EllipticCurve/Q/36a1}{\texttt{36a1}} & $\ZZ/8\ZZ \oplus \ZZ/24\ZZ$ &
\href{http://www.lmfdb.org/EllipticCurve/Q/1849a1}{\texttt{1849a1}} &$\{\mathcal{O}\}$ \\
\href{http://www.lmfdb.org/EllipticCurve/Q/36a2}{\texttt{36a2}} & $\ZZ/8\ZZ \oplus \ZZ/24\ZZ$ &
\href{http://www.lmfdb.org/EllipticCurve/Q/4489a1}{\texttt{4489a1}} & $\{\mathcal{O}\}$ \\
\href{http://www.lmfdb.org/EllipticCurve/Q/49a1}{\texttt{49a1}} & $\ZZ/8\ZZ \oplus \ZZ/8\ZZ$ &
\href{http://www.lmfdb.org/EllipticCurve/Q/26569a1}{\texttt{26569a1}} & $\{\mathcal{O}\}$ \\
\href{http://www.lmfdb.org/EllipticCurve/Q/49a2}{\texttt{49a2}} & $\ZZ/8\ZZ \oplus \ZZ/8\ZZ$ &
&  \\\hline
\end{tabular}
\caption{Torsion structures of CM Elliptic curves}  \label{tab:CMTor}
\end{table}
\renewcommand{\arraystretch}{1}

\section{Determining the possible torsion structures}\label{sec:TorsionStructures}

The last task is to completely determine the finite abelian groups $T$ that arise as the torsion subgroup of an elliptic curve defined over $\QQ$ base-extended to $\QQ(D_4^\infty)$ by determining which combinations of the possible $p$-primary components can occur. Throughout this section we will assume that $E$ is an elliptic curve without complex multiplication, since we have already fully determined the torsion structures for curves with complex multiplication in Table \ref{tab:CMTor}. 

Before we embark on this task, we look back on the previous section and make a useful observation.

\begin{cor}\label{cor:IsogFreeTorsion}
If $E/\QQ$ is an elliptic curve without CM and $p$ is a prime such that $E(\QQ(D_4^\infty))(p)$ is nontrivial, then either $E$ has a $p$-isogeny, or $p = 3$ and $E(\QQ(D_4^\infty)(3) = \ZZ/3\ZZ \oplus \ZZ/3\ZZ$. 
\end{cor}

\subsection{The cases when $13$ or $7$ divides $\#E(\QQ(D_4^\infty))_\tor$ }

We group the case of $13$ and $7$ diving \texorpdfstring{$\#E(\QQ(D_4^\infty))_\tor$}{E(D4)} together because their classifications are similar in spirit. 

\begin{prop}\label{prop:13TorsionClassification}
Let $E/\QQ$ be an elliptic curve such that $13$ divides $\#E(\QQ(D_4^\infty))_\tor$. Then it follows that $E(\QQ(D_4^\infty))_\tor$ is isomorphic to $\ZZ/13\ZZ$. 
\end{prop}

\begin{proof}
From Lemma \ref{lem:isogp^j-k} and Proposition \ref{prop:D4Classification13}, the only way that $13$ can divide $\#E(\QQ(D_4^\infty))_\tor$ is if $E$ has a 13-isogeny. This combined with Corollary \ref{cor:IsogFreeTorsion} and Theorem \ref{thm:IsoTypes} shows that the only other possible nontrivial $p$-primary component is $p=3$. Thus, the last case that needs to be ruled out is the case that $E$ also has full 3-torsion defined over $\QQ(D_4^\infty)$. 

In order to rule out this last possibility we construct the modular curve that would parametrize elliptic curves with full 3-torsion and a point of order 13 defined over $\QQ(D_4^\infty)$. We can construct this modular curve by taking the fiber product of the $j$-maps defined in Propositions \ref{prop:D4Classification3} and \ref{prop:D4Classification13}. For shorthand we will denote the relevant $j$-maps by $j_3(t)$ and $j_{13}(t)$ respectively. The resulting modular curve $X$, is genus 13 with 4 singular points. In order to show that there are no elliptic curves defined over $\QQ$ with full 3-torsion and a point of order 13 defined over $\QQ(D_4^\infty)$ we need to show that there are no nonsingular and noncuspidal rational points on $X$. We can restrict to the nonsingular points here and else where in the paper since the singular points can only correspond to $j=0$ or $j=1728$ and those isomorphism classes have been dealt with in Section \ref{subsec:CMCurves}.

Inspecting $j_3$ and $j_{13}$ we see that they are each invariant under a different fractional linear transformation. In particular,
\[
j_{13}\left(\frac{1}{1-t}\right) = j_{13}(t)\hbox{ and } j_3\left(\frac{-3}{t}\right) = j_3(t).
\]
Using these symmetries, we can construct two distinct automorphisms of the curve $X$ and compute the quotient curve of $X$ by the group $G\subseteq \Aut(X)$, that they generate. Let $\pi\colon X \to X/G = H$ be the quotient map. Then $\pi$ is a surjective map whose image, $H$ is a nonsingular genus 2 hyperelliptic curve. The curve $H$ has a simplified model given by 
\[
H\colon y^2 = x^6 + 10x^3 - 27.
\]
Letting $J$ be the jacobian of $H$, we can use Magma to compute that $J(\QQ)\simeq \ZZ/6\ZZ.$ Thus, in order to find all of the points on $H$ we can use the pullback of the natual embedding of $H(\QQ)$ into $J(\QQ)$. Using Magma we can see that the only points on $H$ that are defined over $\QQ$ are the two points at infinity (there are \emph{two} points at infinity since $H$ has the form $y^2 = f(x)$ with $\deg(f)$ even).
Computing the pullback of these two points at infinity under $\pi$, we see that the only points on $X(\QQ)$ are the 4 singular points.
\end{proof}

\begin{prop}\label{prop:7TorsionClassification}
Let $E/\QQ$ be an elliptic curve such that $7$ divides $\#E(\QQ(D_4^\infty))_\tor$. Then it follows that $E(\QQ(D_4^\infty))_\tor$ is isomorphic to $\ZZ/7\ZZ$. 
\end{prop}

\begin{proof}
Again, from Lemma \ref{lem:isogp^j-k} and Proposition \ref{prop:D4Classification7}, the only way that $7$ can divide $\#E(\QQ(D_4^\infty))_\tor$ is if $E$ has a 7-isogeny. 
The difference this time is that according to Theorem \ref{thm:IsoTypes} there are elliptic curves with 7-isogenies and 2- or 3-isogenies as well. Thus our isogeny argument still leaves open the possibility that $E$ has a point of order 14 or 21 defined over $\QQ(D_4^\infty)$. But, looking at the proof of Proposition \ref{prop:D4Classification7} we see that the only way that $7$ can divide $\#E(\QQ(D_4^\infty))_\tor$ is if $E$ has a 7-isogeny defined over a quadratic field. From \cite[Table 4]{lozano1} we can see that up to $\Qbar$-isomorphism there are exactly 2 elliptic curves with 14-isogenies and 4 elliptic curves with 21-isogenies. Checking these curves by hand we see that none of them have a point of order 7 defined over a quadratic field and thus in these cases 7 doesn't divide $\#E(\QQ(D_4^\infty))_\tor$. 

From Corollary \ref{cor:IsogFreeTorsion}, the only other possibility is that $E$ has its full 3-torsion defined over $\QQ(D_4^\infty)$. Let $j_7(t)$ and $j_3(t)$ be the relevant $j$-maps and construct the curve given by the affine equation defined by the numerator of $j_3(x) - j_7(y)$. The curve $X$ that is defined by this equation is a singular genus 7 curve, but again we can construct two automorphisms from the fact that 
\[
j_{7}\left(\frac{1}{1-t}\right) = j_{7}(t)\hbox{ and } j_3\left(\frac{-3}{t}\right) = j_3(t).
\]
Taking the quotient of $X$ by the subgroup generated by these two automorphisms leaves a hyperelliptic curve $H$ with genus 2. The curve $H$ is given by the simplified equation 
\[
H\colon y^2 = x^6 + 26x^3 - 27.
\]
Letting $J$ be the jacobian of $H$ we use Magma to compute that $J(\QQ) \simeq \ZZ/2\ZZ\oplus\ZZ/6\ZZ$ and again using the pullback of the natural embedding of $H(\QQ)$ into $J(\QQ)$ we get that
\[
H(\QQ) = \{ (-3 , 0), (1 , 0), \infty_+,\infty_- \}.
\]
Computing the pullback of these points through the surjective quotient map, we get that the only points in $X(\QQ)$ are the singular ones. Therefore, there are no elliptic curves with a point of order 7 and full 3-torsion defined over $\QQ(D_4^\infty)$. 
\end{proof}

\begin{remark}
Some of these computations could have been avoided by using the modular interpretation of the quotient curves in Proposition \ref{prop:13TorsionClassification} and \ref{prop:7TorsionClassification}. In those Propositions, the curves points on quotient curves $H$ corresponds to elliptic curves who mod 13 (and respectively mod 7) Galois representations are nonsplit, but the image of the mod 3 representation is contained inside of the 2-Sylow subgroup of $\GL_2(\FF_3)$. Since there are no such curves, there are no rational points on each $H$. 
\end{remark}

\subsection{The case when $5$ divides \texorpdfstring{$\#E(\QQ(D_4^\infty)_\tor$}{E(D4)}}

\begin{prop}\label{prop:Full5TorsionClassification}
Suppose that $E/\QQ$ is an elliptic curve such that $E(\QQ(D_4^\infty))_\tor$ contains a subgroup isomorphic to $\ZZ/5\ZZ \oplus \ZZ/5\ZZ$. Then $E(\QQ(D_4^\infty))_\tor$ is isomorphic to $\ZZ/5\ZZ \oplus \ZZ/5\ZZ$. 
\end{prop}

\begin{proof}
We start by remarking that Corollary \ref{cor:Nontriv5ImplesIsog} says that $E/\QQ$ has its full 5-torsion defined over $\QQ(D_4^\infty)$ if and only if $E$ has two independent 5-isogenies. So first suppose towards a contradiction that $E$ has an $n$-isogeny independent of its two 5-isogenies for some $n>1$. 

If this were the case, the isogeny graph of $E$ would have to contain a subgraph of the form
\[
\xymatrix{
E_1\ar@{<->}[r]^5 \ar@{<->}[d]_n & E \ar@{<->}[r]^5 \ar@{<->}[d]_n & E_2\ar@{<->}[d]^n \\
\widetilde{E}_1  & \widetilde{E} & \widetilde{E}_2 \\  
}
\]
where the number associated to each arrow is the degree of the isogeny. This would imply that there is a $25n$-isogeny between $\widetilde{E}_1$ and $E_2$, but looking at Theorem \ref{thm:IsoTypes} we see that this would force $n$ to be 1, giving us our contradiction. Thus, if $E$ has its full 5-torsion defined over $\QQ(D_4^\infty)$ it can't have any isogenies other than its two independent 5-isogenies. 

From Corollary \ref{cor:IsogFreeTorsion} the last case to rule out is that $E/\QQ$ has its full 15-torsion defined over $\QQ(D_4^\infty)$. To do this let $j_5$ and $j_3$ be the $j$-maps associated with having full 3- and 5-torsion defined over $\QQ(D_4^\infty)$ given in Propositions \ref{prop:D4Classification5} and \ref{prop:D4Classification3}. Next, define $X$ to be the curve given by the affine equation $j_3(x) = j_5(y)$. This time the corresponding curve $X$ is singular with genus 9 and while we are able to compute two separate automorphisms as before, Magma can only easily quotient by one of them at a time. After quotienting out by the first automorphism, we are left with a curve $C$ of genus 4. In this case Magma is able to compute $\Aut(C)$ in a reasonable amount of time, and we get that $\Aut(C) \simeq \ZZ/2\ZZ\oplus\ZZ/2\ZZ$. If we quotient out by the full automorphism group of $C$ we are left with a genus 0 curve and this is not useful. Instead, we pick a subgroup of $\Aut(C)$ of order 2 to quotient out by and we are left with a genus 1 elliptic curve $F/\QQ$. The curve that Magma gives back has large coefficients (the smallest one has over 1000 digits) so it is cumbersome to work with, but with some work we are able to show that $F$ is isomorphic to the curve with Cremona label \href{http://www.lmfdb.org/EllipticCurve/Q/15a8}{\texttt{15a8}} over $\QQ$. The Mordel--Weil group over $\QQ$ of \href{http://www.lmfdb.org/EllipticCurve/Q/15a8}{\texttt{15a8}} is isomorphic to $\ZZ/4\ZZ$, and so we find a point of order 4 on $F$ defined over $\QQ$ and pull its multiples back through the reduction maps to see that $X(\QQ)$ only contains singular points. 
\end{proof}

\begin{prop}\label{prop:C53TorsionClassification}
Suppose that $E/\QQ$ is an elliptic curve such that the $5$-primary component of $E(\QQ(D_4^\infty))_\tor$ is isomorphic to $\ZZ/5\ZZ$. Then the $3$-primary component of $E(\QQ(D_4^\infty))_\tor$ is either trivial or isomorphic to $\ZZ/3\ZZ$ or  $\ZZ/3\ZZ \oplus \ZZ/3\ZZ$. 
\end{prop}

\begin{proof}
Let $E$ be an elliptic curve such that the $5$-primary component of $E(\QQ(D_4^\infty))_\tor$ is isomorphic to $\ZZ/5\ZZ$. This means that $E$ has a rational 5-isogeny and while we expect that generically this is the only isogeny $E$ has, it is possible that $E$ does in fact also have a 3-isogeny. In this case, $E(\QQ(D_4^\infty))_\tor$ will automatically contain a point of order 15. Since there are no elliptic curves with 45-isogenies, the last possibility for the 3-primary component of $E(\QQ(D_4^\infty))_\tor$ is that it is isomorphic to $\ZZ/3\ZZ\oplus\ZZ/3\ZZ$.

Again, we start by constructing the associated modular curve using the relevant $j$-maps, this time getting a singular curve $X$ that is genus 1. Finding a nonsingular model for this curve, shows that it is isomorphic over $\QQ$ to the elliptic curve $F$ with Cremona reference \href{http://www.lmfdb.org/EllipticCurve/Q/15a8}{\texttt{15a3}}. From \cite{lmfdb} we see that $F(\QQ)\simeq \ZZ/2\ZZ\oplus\ZZ/4\ZZ$ and pulling these points back to $X$ we get that the nonsingular points of $X(\QQ)$ are exactly 
$$\{(9/2 , -5/8 ), (9/32 , -8/25 ), (-32/3 , -8/25 ), (-2/3 , -5/8 )  \}.$$
Plugging the appropriate coordinates into the corresponding $j$-maps, we see that there are exactly 2 elliptic curves up to $\Qbar$-isomorphism such that $E(\QQ(D_4^\infty))_\tor$ contains a subgroup isomorphic to $\ZZ/3\ZZ\oplus\ZZ/15\ZZ$ and they have 
\[
j(E) = -\frac{1680914269}{32768}\hbox{ or }j(E) = \frac{1331}{8}.
\]
\end{proof}

\begin{prop}\label{prop:C52TorsionClassification}
Suppose that $E/\QQ$ is an elliptic curve such that the $5$-primary component of $E(\QQ(D_4^\infty))_\tor$ is isomorphic to $\ZZ/5\ZZ$. Then the $2$-primary component of $E(\QQ(D_4^\infty))_\tor$ is either trivial or isomorphic to $\ZZ/4\ZZ \oplus \ZZ/4\ZZ$. 
\end{prop}

\begin{proof}
Let $E$ be an elliptic curve such that the $5$-primary component of $E(\QQ(D_4^\infty))_\tor$ is isomorphic to $\ZZ/5\ZZ$. This means that $E$ has a rational $5$-isogeny and if this is the only isogeny that $E$ has, then the 2-primary component of $E(\QQ(D_4^\infty))$ must be trivial. Inspecting Theorem \ref{thm:IsoTypes}, we see that it is possible that $E$ could also have a 2-isogeny, in which case the full 4-torsion of $E$ would be defined over $\QQ(D_4^\infty)$. All that is left to show is that there are no curves that have a 5-isogeny and two primary component of $E(\QQ(D_4^\infty))_\tor$ containing a subgroup isomorphic to $\ZZ/4\ZZ \oplus \ZZ/8\ZZ$.

According to Table \ref{tab:RZBData}, Figure \ref{fig:RZB}, and \cite{RZB} there are four ways that $E(\QQ(D_4^\infty))_\tor$ can contain a subgroup isomorphic to $\ZZ/4\ZZ\oplus \ZZ/8\ZZ$. These 4 possibilities break in to pairs. The first pair is that $E$ comes from a points on either $X_{12}$ or $X_{13}$ which can be interpreted as the following two conditions:
\begin{enumerate}
\item the curve $E$ has a point of order 2 defined over $\QQ$ and $\Delta(E) \equiv \Delta(E_2)\bmod (\QQ^\times)^2$ where $E_2$ is the elliptic curve that is 2-isogenous to $E$, or
\item the curve $E$ has a 4-isogeny. 
\end{enumerate}
Since we are already assuming that $E$ has a $5$-isogeny, $(2)$ cannot happen by Theorem \ref{thm:IsoTypes} since there are no elliptic curves over $\QQ$ with a 20-isogeny. Further, if $E$ has a 5-isogeny and satisfies condition (1) then $E$ is isomorphic to
\begin{align*}
E_t\colon y^2 + xy &= x^3 - \frac{36(t-4)(t+1)^2t^5}{(t^2-2t-4)^2(t^2-2t+2)^2(t^2+4)(t^4-2t^3-6t^2-8t-4)}x\\
&-\frac{(t-4)(t+1)^2t^5}{(t^2-2t-4)^2(t^2-2t+2)^2+(t^2+1)(t^4-2t^3-6t^2-8t-4)}.
\end{align*}
over $\QQ(D_4^\infty)$ for some $t\in\QQ$. This curve has discriminant \[\Delta(E_t) = \frac{(t - 4)t^5(t + 1)^2(t^6 - 4t^5 + 16t + 16)^6}{(t^2 - 2t - 4)^6(t^2 - 2t + 2)^6(t^2 + 4)^3(t^4 - 2t^3 - 6t^2 - 8t - 4)^6}\] 
and its 2-isogenous curve has discriminant \[\Delta(E_{t,2}) = \frac{(t - 4)^2t^{10}(t + 1)(t^6 - 4t^5 + 16t + 16)^6}{(t^2 - 2t - 4)^6(t^2 - 2t + 2)^6(t^2 + 4)^3(t^4 - 2t^3 - 6t^2 - 8t - 4)^6}.\]
These two curves have discriminants in the same square class exactly when $(t + 1)/(t^6 - 4t^5)$ is a square. Considering the curve 
\[C\colon s^2 = (t + 1)/(t^6 - 4t^5)\]
in Magma we show that $C$ has genus 1 and exactly 1 rational point that is singular, $P = (-1,0)$. Since the curve $E_0$ is singular here we have that there are no elliptic curves with a 10-isogeny whose 2-isogenous curve has the same discriminant mod squares, i.e. there are no curves with a point of order 5 over $\QQ(D_4^\infty)$ that satisfy condition (1). 

Lastly from Table \ref{tab:RZBData}, Figure \ref{fig:RZB}, and \cite{RZB} there are two ways that $E(\QQ(D_4^\infty))_\tor$ could also correspond to a rational point on $X_8$ or $X_{11}$. These two possibilities correspond to the following two conditions:
\begin{enumerate}
\item The curve $E$ has a point of order 2 defined over $\QQ$ and $\Delta(E) \equiv -\Delta(E_2)\bmod (\QQ^\times)^2$ where $E_2$ is the elliptic curve that is 2-isogenous to $E$.
\item The curve $E$ its full 2-torsion defined over $\QQ$. 
\end{enumerate}
If $E$ has a 5-isogeny, then $E$ cannot have its full 2-torsion defined over $\QQ$ as this would imply the existence of an elliptic curve with a 20-isogeny. So we can rule out condition (2). To show that there are no elliptic curves with a 5-isogeny that satisfy condition (1), we proceed the same way as before, but this time we consider the curve 
\[C'\colon -s^2 = (t + 1)/(t^6 - 4t^5).\]
The curve $C'$ is also genus 1 and singular and it has only one rationla points in the given affine patch which corresponds to a singluar curve. Thus condition (2) cannot occur for an elliptic curve with a 5-isogeny, completing the proof of the proposition. 

\end{proof}

\begin{prop}
Suppose that $E/\QQ$ is an elliptic curve such that the $5$-primary component of $E(\QQ(D_4^\infty))_\tor$ is isomorphic to $\ZZ/5\ZZ$. Then $E(\QQ(D_4^\infty))_\tor\simeq \ZZ/5\ZZ,\ZZ/15\ZZ$, $\ZZ/3\ZZ\oplus\ZZ/15\ZZ$, or  $\ZZ/4\oplus\ZZ/20\ZZ$.
\end{prop}

\begin{proof}
In order to prove this proposition we simply need to show that if the 5-primary component of $E(\QQ(D_4^\infty))_\tor$ is $\ZZ/5\ZZ$ and the 3-primary component of $E(\QQ(D_4^\infty))_\tor$ is nontrivial, then the 2-primary component of $E(\QQ(D_4^\infty))_\tor$ is trivial. 

First, suppose that the 3-primary component of $E(\QQ(D_4^\infty))_\tor$ is nontrivial but $E[3]\not\subseteq E(\QQ(D_4^\infty))$. Then it must be that $E$ has a 15-isogeny and by Theorem \ref{thm:IsoTypes} this means that $E$ cannot also have a 2-isogeny and the 2-primary component of $E(\QQ(D_4^\infty))_\tor$ is trivial. 

Lastly, suppose that the 3-primary component of $E(\QQ(D_4^\infty))_\tor$ is isomorphic to $\ZZ/3\ZZ \oplus \ZZ/3\ZZ$. From Proposition \ref{prop:C53TorsionClassification}, either $j(E) = -\frac{1680914269}{32768}$ or $j(E) = \frac{1331}{8}$. Using Proposition \ref{prop:twist} it is enough to check one curve with each $j$-invariant. Doing so in both cases shows that both curves have torsion subgroup $\ZZ/3\ZZ\oplus\ZZ/15\ZZ$ over $\QQ(D_4^\infty)$. 
\end{proof}

\subsection{The cases when $3$ divides \texorpdfstring{$\#E(\QQ(D_4^\infty))_\tor$}{E(D4)}}

The last thing that we need to determine is what are the possible $2$-primary components of $E(\QQ(D_4^\infty))_\tor$ give that the $3$-primary component is nontrivial. This breaks into $3$ different cases that correspond to the $3$ possible nontrivial $3$-primary components listed in Proposition \ref{prop:D4Classification3}.

\begin{prop}
Suppose that $E/\QQ$ is an elliptic curve such that the $3$-primary component of $E(\QQ(D_4^\infty))_\tor$ is isomorphic to $\ZZ/3\ZZ$. Then the $2$-primary component of $E(\QQ(D_4^\infty))_\tor$ is either trivial or isomorphic to either $\ZZ/4\ZZ \oplus \ZZ/4\ZZ$, $\ZZ/4\ZZ \oplus \ZZ/8\ZZ$, or $\ZZ/8\ZZ \oplus \ZZ/8\ZZ$. 
\end{prop}

\begin{proof}
From Lemma \ref{lem:isogp^j-k} that $E$ must have a $3$-isogeny and since we have examples of the three torsion structures listed above, all that is left to do is rule out the case where the 2-primary component of $E(\QQ(D_4^\infty))$ contains a point of order 16. Looking at Table \ref{tab:RZBData} and the data in \cite{RZB} we see that in order for $E$ to have a point of order 16 defined over $\QQ(D_4^\infty)$ it must be that either $E$ has an 8-isogeny or it must have a 4-isogeny and an independent 2-isogeny. In both of these cases, since $E$ also has a 3-isogeny there would have to exist an elliptic curve (either $E$ itself or another elliptic curve in its isogeny class) that has a 24-isogeny. Once again thanks to Theorem \ref{thm:IsoTypes} we have that this can't happen, and thus if the 3-primary component of $E(\QQ(D_4^\infty))_\tor$ is isomorphic to $\ZZ/3\ZZ$, then $E$ cannot have a point of order 16 defined over $\QQ(D_4^\infty)$. 
\end{proof}

\begin{prop}
Suppose that $E/\QQ$ is an elliptic curve such that the $3$-primary component of $E(\QQ(D_4^\infty))_\tor$ is isomorphic to $\ZZ/9\ZZ$. Then the $2$-primary component of $E(\QQ(D_4^\infty))_\tor$ is trivial. 
\end{prop}

\begin{proof}
Recall from Proposition \ref{prop:D4Classification3} that $E$ has a point of order $9$-defined over $\QQ(D_4^\infty)$ if and only if $E$ has a quadratic twist with a point of order 9 defined over $\QQ$. Since every quadratic extension of $\QQ$ is contained inside of $\QQ(D_4^\infty)$, we can assume that we are working with a model of $E/\QQ$ such that $E$ has a point of order 9 defined over $\QQ$ itself. The only way that the 2-primary component of $E(\QQ(D_4^\infty))_\tor$ is nontrivial is for $E$ to have a rational point of order 2. Thus, in order for $E(\QQ(D_4^\infty))_\tor$ to have a point of order 9 \emph{and} nontrivial 2-primary component, it would have to be that $E$ has a rational point of order 18, but by Theorem \ref{thm:mazur} this cannot happen. 
\end{proof}

\begin{prop}
Suppose that $E/\QQ$ is an elliptic curve such that the $3$-primary component of $E(\QQ(D_4^\infty))_\tor$ is isomorphic to $\ZZ/3\ZZ \oplus \ZZ/3\ZZ$. Then the $2$-primary component of $E(\QQ(D_4^\infty))_\tor$ is either trivial or isomorphic to either $\ZZ/4\ZZ \oplus \ZZ/4\ZZ$, or $\ZZ/4\ZZ \oplus \ZZ/8\ZZ$. 
\end{prop}

\begin{proof}
It is sufficient to show that if $E$ has its full 3-torsion defined over $\QQ(D_4^\infty)$ then it can't have its full 8-torsion or a point of order 16 defined over $\QQ(D_4^\infty)$. 

Assume that $E$ is an elliptic curve such that $E$ has its full 3- and full 8-torsion defined over $\QQ(D_4^\infty)$. From Proposition \ref{prop:D4Classification3} and Table \ref{tab:RZBData} there are two ways in which this can happen. These two possibilities correspond to two different modular curves that can be again constructed setting the two relevant $j$-maps equal to each other. These curves both end up begin singular curves with genus 1 and we can use Magma to show that they both only contain singular points which can only correspond to the curve with $j=0$ or $j=1728$. Thus giving us a contradiction. 

Next, assume that $E$ is an elliptic curve with its full 3-torsion and a point of order 16 defined over $\QQ(D_4^\infty)$. Since we have shown that $E$ cannot have its full 8-torsion defined over $\QQ(D_4^\infty)$ the 2-primary component of $E(\QQ(D_4^\infty))_\tor$ must be either $\ZZ/4\ZZ\oplus\ZZ/16\ZZ$ or $\ZZ/4\ZZ\oplus\ZZ/32\ZZ$. Again, using the data from \cite{RZB} we see that in order for this to be the case $E$ must have an 8-isogeny. Constructing the modular curve $X$ parametrizing elliptic curves with an 8-isogeny and full 3-torsion over $\QQ(D_4^\infty)$, we get a genus 3 singular curve. It turns out that $X$ is a hyperelliptic curve whose jacobian has rank 0, but since much of Magma's functionality is only implemented for genus 2 hyperelliptic curves, we take the extra step to quotient out by a subgroup of $\Aut(X)$ generated by a single automorphism of order 2. The resulting quotient curve $H$ is hyperelliptic, genus 2 and given by the equation $H\colon y^2 = x^6+1$. The jacobian of $H$ again has rank 0 over $\QQ$ and finding all the points on $H$ in Magma and pulling them back to $X$, we se that $X(\QQ)$ only contains singular points and points where both $j$-maps are undefined. 
\end{proof}

\subsection{When only 2 divides \texorpdfstring{$\#E(\QQ(D_4^\infty))_\tor$}{E(D4)}.}

When $\#E(\QQ(D_4^\infty))_\tor$ is a nontrivial power of 2, we know from Table \ref{tab:RZBData} that 
\[E(\QQ(D_4^\infty)_\tor \simeq \begin{cases}
 \ZZ/4\ZZ\oplus\ZZ/2^j\ZZ & j= 2,3,4,5,\hbox{ or}\\
 \ZZ/8\ZZ\oplus\ZZ/2^j\ZZ & j= 3,4,5,\hbox{ or}\\
 \ZZ/16\ZZ\oplus\ZZ/16\ZZ. &\\
\end{cases}\]
Further, the examples listed in Table \ref{tab:Examples} show that each of these possibilities does occur. 

\section{Parameterizations for each possible torsion structure.}\label{sec:Tparam}
Since the torsion structure of $E/\QQ$ base-extended to $\QQ(D_4^\infty)$ only depends on the $j$-invariant of $E$ (unless $j(E) = 0$) to completely parametrize when each torsion structure occurs, it is sufficient to give explicit descriptions of the sets 
\[ 
S_T = \{ j(E) : E(\QQ(D_4^\infty))_\tor \simeq T \},
\]
where $T$ ranges over the 24 possible torsion structures determined in Section \ref{sec:TorsionStructures} and listed in Table \ref{tab:Examples}. From Proposition \ref{prop:twist} if $j(E) \neq 0$, then $j(E)$ is in exactly one set $S_T$.  Further, Lemma \ref{lem:j=0} shows $j(E) = 0$ is in 3 different sets $S_T$ depending which rational model is chosen; namely it is in the sets $S_T$ for $T = \ZZ/3\ZZ$, $\ZZ/3\ZZ\oplus\ZZ/3\ZZ$, and $\ZZ/8\ZZ \oplus\ZZ/24\ZZ$. 

We will describe each set $S_T$ by providing sets $F_T$ of (possibly constant) rational functions $j(t)$ which parameterize the $j$-invariants $j(E)$ of elliptic curves $E/\QQ$ for which $E(\QQ(D_4^\infty))$ contains a subgroup isomorphic to $T$. In order to ease notation, we will let $\mathcal{T}$ be the set of the 24 possible torsion structures for $E(\QQ(D_4^\infty))$ and we will put a partial order on $\mathcal{T}$ given by $T_1\leq T_2$ exactly when $T_2$ has a subgroup isomorphic to $T_1$. Further, for a fixed elliptic curve $E/\QQ$ with $j(E) \neq 0$ we let $\mathcal{T}(E) \subseteq \mathcal{T}$ be the set of groups $T$ for which $j(E)$ is in the image of some function in $F_T$. 

\begin{table}[h!]
\begin{center}
\renewcommand{\arraystretch}{1.6}
\begin{tabular}{l|l}
$T$&$j(t)$\\\hline
$\{\mathcal{O}\} $ &  $t$   \\
$\ZZ/3\ZZ $ &  $\frac{27(t+1)(t+9)^3 }{t^3}$   \\
$\ZZ/5\ZZ $ & $\frac{5^2(t^2+10t+5)^3}{t^5}$    \\
$\ZZ/7\ZZ $ &   $\frac{(t^2-t+1)^3(t^3-8t^2+5t+1)(t^6-11t^5+30t^4-15t^3-10t^2+5t+1)^3}{(t-1)^7t^6}$  \\
$\ZZ/9\ZZ $ & $\frac{(t^3 - 3t^2 + 1)^3  (t^9 - 9t^8 + 27t^7 - 48t^6 + 54t^5 - 45t^4 + 27t^3 - 9t^2 + 1)^3}{(t-1)^9t^9(t^2-t+1)^2(t^3-6t^2+3t+1) }$    \\
$\ZZ/13\ZZ $ &  $\frac{(t^2-t+1)^3( t^{12}-9t^{11}+29t^{10}-40t^9+22t^8-16t^7+40t^6-22t^5-23t^4+25t^3-4t^2-3t+1)^3}{(t-1)^{13}t^{13}(t^3-4t^2+t+1)}$   \\
$\ZZ/15\ZZ $ &  $\{-\frac{25}{2}, -\frac{349938025}{8}, -\frac{121945}{32}, \frac{46969655}{32768}\}$   \\
$\ZZ/3\ZZ \oplus \ZZ/3\ZZ $ &  $\frac{27(t+1)^3(t-3)^3 }{t^3}$   \\
$\ZZ/3\ZZ \oplus \ZZ/15\ZZ $ &   $\{ -\frac{1680914269}{32768}, \frac{1331}{8}\}$  \\
$\ZZ/4\ZZ \oplus \ZZ/4\ZZ $ & $\frac{t^3}{t + 16}$    \\
$\ZZ/4\ZZ \oplus \ZZ/8\ZZ $ & $\frac{(t^2 + 16)^3}{t^2}$, $\frac{(t^2 - 48)^3}{(t - 8)(t + 8)}$    \\
$\ZZ/4\ZZ \oplus \ZZ/12\ZZ $ &  $\frac{(t+6)^3(t^3+18t^2 + 84t +23)^2}{t(t+8)^3(t+9)^2}$   \\
$\ZZ/4\ZZ \oplus \ZZ/16\ZZ $ & $\frac{(t^4 - 16t^2 + 16)^3}{(t - 4)t^2(t + 4)}$    \\
$\ZZ/4\ZZ \oplus \ZZ/20\ZZ $ &  $\frac{(t^6-4t^5+16t+16)^3}{(t+1)^2(t-4)t^5}$   \\
$\ZZ/4\ZZ \oplus \ZZ/24\ZZ $ &  $\frac{(t^2-3)^3(t^6-9t^4+3t^2-3)^3}{t^4(t^2-9)(t^2-1)^3}$   \\
$\ZZ/4\ZZ \oplus \ZZ/32\ZZ $ & $\frac{(t^{16} - 8t^{14} + 12t^{12} + 8t^{10} - 10t^8 + 8t^6 + 12t^4 - 8t^2 + 1)^3}{(t-1)^4  t^{16} (t+1)^4  (t^2 - 2t - 1)  (t^2+1)^2  (t^2 + 2t - 1)}$ \\
$\ZZ/5\ZZ \oplus \ZZ/5\ZZ $ &   $\frac{(t^2+5t+5)^3(t^4+5t^2+25)^3(t^4+5t^3+20t^2+25t+25)^3}{t^5(t^4+5t^3+15t^2+25t+25)^5}$  \\
$\ZZ/8\ZZ \oplus \ZZ/8\ZZ $ &  $\frac{(t^2 + 3)^3}{(t - 1)^2(t + 1)^2}$, $\frac{(t - 4)^3(t + 4)^3}{t^2}$   \\
$\ZZ/8\ZZ \oplus \ZZ/16\ZZ $ & $\frac{(t^4 - 8t^3 + 2t^2 + 8t + 1)^3(t^4 + 8t^3 + 2t^2 - 8t + 1)^3}{(t - 1)^2t^2(t + 1)^2(t^2 + 1)^8}$    \\
$\ZZ/8\ZZ \oplus \ZZ/24\ZZ $ & $\frac{(t^2 - 6t + 21)^3(t^6 - 18t^5 + 75t^4 + 180t^3 - 825t^2 - 2178t + 6861)^3}{(t - 9)^2(t - 5)^6(t - 3)^2(t - 1)^6(t + 3)^2}$    \\
$\ZZ/8\ZZ \oplus \ZZ/32\ZZ $ &  $\frac{(65536t^{16} - 131072t^{14} + 49152t^{12} + 8192t^{10} + 58880t^8 + 512t^6 + 192t^4 - 32t^2 + 1)^3}{256t^8(2t-1)^8(2t+1)^8 (4t^2-4t-1)^2(4t^2+1)^4(4t^2+4t-1)^2}$\\
$\ZZ/12\ZZ \oplus \ZZ/12\ZZ $ &  $\frac{(t^3 - 15t^2 - 33t + 955)^3(t^3 - 15t^2 + 75t - 233)^3}{729(t-11)^3(t-8)^6(t+1)^3}$   \\
$\ZZ/12\ZZ \oplus \ZZ/24\ZZ $ &   $\{8000\}$  \\
$\ZZ/16\ZZ \oplus \ZZ/16\ZZ $ & $\frac{(t^4 - 2t^3 + 2t^2 + 2t + 1)^3(t^4 + 2t^3 + 2t^2 - 2t + 1)^3}{(t - 1)^4t^4(t + 1)^4(t^2 + 1)^4}$    \\\hline
\end{tabular}
\end{center}
\caption{Parameterizations $j(t)$ of the $\Qbar$-isomorphism classes of elliptic curves $E/\QQ$ according to isomorphism type of $E(\QQ(D_4^\infty))_\tor$.}  \label{tab:jMaps}
\end{table}

\begin{thm}\label{thm:jmaps}
Let $E/\QQ$ be an elliptic curve with $j(E) \neq 0$. Then, the set $\mathcal{T}(E)$ has a unique maximal element $T(E)$ with respect to the partial order on $\mathcal{T}$, and $E(\QQ(D_4^\infty))_\tor$ is isomorphic to $T(E)$. 
\end{thm}

\begin{remark}
The set $\mathcal{T}(E)$ need not contain every $T \leq T(E)$.  For example, the curve \href{http://www.lmfdb.org/EllipticCurve/Q/338e1}{\texttt{338e1}} has $T(E) = \ZZ/3\ZZ \oplus \ZZ/3\ZZ$, but $j(E)$ is not in the image of the unique function $j(t)$ for $\ZZ/3\ZZ$ since it does not have a rational 3-isogeny. 
\end{remark}

The proof of Theorem \ref{thm:jmaps} follows from the exact same argument that is presented in the proof of \cite[Theorem 7.1]{Q(3)} changing a few minor details. For the sake of brevity, we omit its proof here and instead finish by justifying the $j$-maps that appear in Table \ref{tab:jMaps}.
\begin{itemize}
\item Justification for any of the functions for a torsion structure $T$ of the form $\ZZ/p^k\ZZ \oplus \ZZ/p^j\ZZ$, can be found in Section \ref{sec:pPrimary}.
\item $\ZZ/15\ZZ$: In section \ref{sec:pPrimary}, we saw that in order for $E/\QQ$ to have a point of order 3 nd not full level 3 structure or 5 defined over $\QQ(D_4^\infty)$ it is necessary and sufficient for $E$ to have a rational 3- or 5-isogeny respectively. Thus, $E(\QQ(D_4^\infty))_\tor$ is isomorphic to $\ZZ/15\ZZ$ exactly when $E$ has a 15-isogeny and there are exatly 4 $\Qbar$-isomorphism classes such curves. See \cite[Table 4]{lozano1}.
\item $\ZZ/3\ZZ\oplus\ZZ/15\ZZ$: In Section \ref{sec:pPrimary} that $E/\QQ$ has its full 3-torsion defined over $\QQ(D_4^\infty)$ exactly when the mod 3 Galois representation associated to $E$ has its image contained in the normalizer of a nonsplit Cartan subgroup. Thus, $\mathcal{T}(E) = \ZZ/3\ZZ \oplus\ZZ/15\ZZ$ if and only if $\Im\rho_{E,3}$ is conjugate to a subgroup of the normalizer of the split Cartan subgroup of $\GL_2(\ZZ/3\ZZ)$ and $E$ has a 5-isogeny. We can construct the curve $X$ that parameterizes such elliptic curve as the fiber product of the modular curves $X^+_{s}(3)$ and $X_0(5)$. The resulting curve is a genus 1 singular curve whose desingularization is the elliptic curve \href{http://www.lmfdb.org/EllipticCurve/Q/15a3}{\texttt{15a3}}. The curve \href{http://www.lmfdb.org/EllipticCurve/Q/15a3}{\texttt{15a3}} has rank zero and torsion group $\ZZ/2\ZZ\oplus \ZZ/4\ZZ$ over $\QQ$. Computing the pullback of these 8 points back to $X$ in Magma and then applying the $j$-maps from $X$ to $\PP^1$ we see that there are only 2 possible $j$-invariant which are listed in Table \ref{tab:jMaps}.
\item $\ZZ/4\ZZ\oplus\ZZ/12\ZZ$: In Section \ref{sec:pPrimary} we showed that for each of these $p$-primary components to occur it is necessary and sufficient for $E$ to have a rational point of order 2 and a rational 3-isogeny. This is the same as saying that $E$ has a rational 6-isogeny since having a point of order 2 is equivalent to have a 2-isogeny. These curves are parameterized by the genus zero curve $X_0(6)$ whose $j$-map has been taken from \cite[Table 3]{lozano1}.
\item $\ZZ/4\ZZ\oplus\ZZ/20\ZZ$: As in the previous case, Section \ref{sec:pPrimary} shows that these two $p$-primary components occur exactly when $E$ has a rational point of order 2 and a 5-isogeny. Again, this is equivalent to $E/\QQ$ having a 10-isogeny and such curves are parameterized by $X_0(10)$ whose $j$-map is taken from \cite[Table 3]{lozano1}.
\item $\ZZ/4\ZZ\oplus\ZZ/24\ZZ$: This time there are 2 possible ways for this torsion structure to occur and both of them require that $E$ has a 3-isogeny. The two possibilities correspond to the two distinct ways for $E$ to have 2-primary component $\ZZ/4\ZZ\oplus\ZZ/8\ZZ$ (see Table \ref{tab:RZBData}). 
\begin{itemize}
\item The first possibility is that $E$ has a 2-isogeny and that the discriminant of $E$ is equivalent to the discriminant of the 2-isogenous curve modulo squares. Starting with a generic elliptic curve with a 6-isogeny in Magma and computing the 2-isogenous curve we see that there are no elliptic curves over $\QQ$ with a 6-isogeny such that its discriminant is equivalent to that of its two isogenous curve modulo squares.
\item The second possibility is that $E$ has a 4-isogeny. Combining this with a 3-isogeny, $E$ has torsion structure $\ZZ/4\ZZ\oplus\ZZ/24\ZZ$ if and only if $E$ has a 12-isogeny. Again these curves are parameterized by the genus 0 modular curve $X_0(12)$ whose $j$-map is taken from \cite[Table 3]{lozano1}.
\end{itemize}
\item $\ZZ/8\ZZ\oplus\ZZ/24\ZZ$: Again, there are two distinct ways that $E$ can have its full 8-torsion defined over $\QQ(D_4^\infty)$ (see Table \ref{tab:RZBData}) so this breaks into two distinct cases. 
\begin{itemize}
\item The first possibility is that $E$ has a 2-isogeny and the discriminant of $E$ is equivalent to the \emph{negative} of the discriminant of its two isogenous curve. Just as before, we start with a generic elliptic curve with a 6-isogeny and see that it is impossible for such an elliptic curve to be defined over $\QQ$. 
\item The second possibility is that $E$ has its full 2-torsion defined over $\QQ$. In this case, any twist of $E$ also will have its full 2-torsion defined over $\QQ$ and since any elliptic curve with a 3-isogeny has a quadratic twist with a rational point of order $3$, we see that in this case $E$ must have a quadratic twist with torsion subgroup isomorphic to $\ZZ/2\ZZ\oplus\ZZ/6\ZZ$. These curves have been completely parametrized and the $j$-map has been taken from \cite[Figure 2]{lozano2}.
\end{itemize}
\item $\ZZ/12\ZZ\oplus\ZZ/12\ZZ$: This can only occur if $E$ has a point of order 2 defined over $\QQ$ and $E$ is nonsplit at 3. Constructing the modular curve $X$ that parameterizes such elliptic curves by considering the fiber product of the genus 0 modular curves $X_1(2)$ and $X_{s}^+(3)$ we see that the resulting curve has genus 0. The $j$-map of this curve is computed using Magma. \item $\ZZ/12\ZZ\oplus\ZZ/24\ZZ$: Once again there are two cases to consider corresponding to the two distinct ways that $E$ can have 2-primary component $\ZZ/4\ZZ\oplus\ZZ/8\ZZ$. Constructing the fiber product of the two corresponding curves with $X_{ns}^+(3)$ we get two different genus 1 modular curves. In total, these two curves only contain 1 nonsingular and noncuspidal point defined over $\QQ$. This point corresponds to the elliptic curve with $j = 8000$.
\end{itemize}

\end{document}